\documentclass[11pt]{amsart}

\usepackage{mathrsfs, bbm,amssymb} 
\usepackage{mathabx}

\newtheorem{theorem}{Theorem}
\newtheorem{lemma}[theorem]{Lemma}
\newtheorem{corollary}[theorem]{Corollary}

\theoremstyle{definition}

\newtheorem{proposition}[theorem]{Proposition}
\newtheorem{remark}[theorem]{Remark}

\numberwithin{equation}{section}
\numberwithin{theorem}{section}

\usepackage{color,hyperref}
\definecolor{darkblue}{rgb}{0.0,0.0,0.85}
\definecolor{darkgreen}{rgb}{0,0.5,0}
\hypersetup{colorlinks,breaklinks,
            linkcolor=darkblue,urlcolor=darkblue,
            anchorcolor=darkblue,citecolor=darkblue}
\newcommand\mydoi[1]{\href{http://dx.doi.org/#1}{doi:#1}}
\newcommand\myurl[1]{\href{#1}{#1}}
    
\DeclareMathOperator*{\esssup}{sup}
\def\esup{\esssup_{[0,T]}}
 
\def\eesup{\sup_{[0,T]}}

\def\sX{{\mathsf X}}
 \def\H{{\mathsf H}}
\def\sL{{\mathsf L}}

\def\LL{{{\mathsf L}^2}}
\def\Linf{{\mathsf L}^\infty}
  
\def\cL{{\mathcal L}}

\def\cN{\mathcal N}
 \def\cP{{\mathcal P} } 
\def\cQ{{\mathcal  Q}}  
 \def\bC{{\mathbf C}}
 
 \def\scrH{\mathscr H}
 \def\scrB{\mathscr B}
 
\def\rb{{\mathrm  b}}
\def\bB{\rb^{vor}}

 \def\su{u}
 \def\sv{v}
 
 \def\mR{{\mathbb R}}   
  \def\ep{\epsilon}
  
\newcommand{\rB}[1]{ {\mathrm B} \bigl(#1\bigr)}
\newcommand{\nmm}[1]{\vvvert #1 \vvvert}
\newcommand{\mm}[2]{\| #1 \|_{\H^{#2}}}
\newcommand{\wL}[1]{\| #1\|_{{{\mathsf L}^2_\DM} }}

 \def\DM{{\sigma}}
\def\Din{\DM^{-1}}
\def\Vk{\ptau^kV} 
\def\Vkl{\ptau^{k+1}V}
\def\Vm{\ptau^mV} 
\def\Vml{\ptau^{m-1}V}
  
 \def\vv{{\mathbf v}}
\def\vq{\vv^{\text{\tiny$Q$}}}
\def\vp{\vv^{\text{\tiny$P$}}}
\def\vl{\widetilde{\vv}}
  \def\brv{\breve{\vv}}
\def\brom{\breve{\om}}
\def\brr{\breve{r}}
\def\brR{{R}_{\pres} }
\def\om{\omega}

\def\si{\gamma}
 \def\th{\theta}
 \def\thl{\widetilde{\theta}}

 \def\bc{\big|_{\pa\Omega}}
\def\nm{{\widearrow{{\mathsf{n}}}}}
 \def\tot{\mathsf{tot}}
\def\pres{{\pi}} 
\def\hep{{h}_\ep}
\def\Eto{\mathrm{E}_{t,0}}
\def\Eo{\mathrm{E}_{0}}

\newcommand{\ave}[1]{{\overline{#1}}} 

\def\intT{\displaystyle\int_0^T}
\def\intO{\displaystyle\int_\Omega}
\def\ttdiff{\Big|_{\tau_1}^{\tau_2} }
\def\ttint{\int_{\tau_1}^{\tau_2}}

\def\ccdot{{\hspace{-0.15mm}\cdot }}

\def\nc{\nabla \ccdot }
\def\cn{ \!\cdot\!\! \nabla}
 \def\nt{\nabla\!\times\!}
\def\pa{\partial}
\def\pt{\pa_t} 
\def\px{\pa_x} 
\def\ptau{\pa_{\tau}}

 \def\be{\begin{equation}}
\def\ee{\end{equation}}
 \def\bpm{\begin{pmatrix}}
\def\epm{\end{pmatrix}}

\makeatletter
\@ifpackageloaded{amssymb}{}{ \def\lesssim{\stackrel{{}_<}{{}_\sim}}}
\makeatother
 \def\lsm{\lesssim}

\newcommand{\rev}[1]{#1}

 \title[Accuracy of Incompressible Approximation]{Improved Accuracy of Incompressible Approximation of Compressible Euler   Equations}

\author[Bin Cheng]{Bin Cheng}
 \address{\newline
       Department of Mathematics\\
University of Surrey, 
Guildford, GU2 7XH\\
United Kingdom }
\email[Bin Cheng]{b.cheng@surrey.ac.uk}
 
\keywords{Incompressible approximation; compressible Euler equations;  convergence rates;  time averages; initial-boundary value problem; singular limits; multiscale PDEs; hyperbolic PDEs.}
\subjclass{35Q31 (Primary) 35L50, 76N15 (Secondary)}

\begin{document}
\maketitle 
\begin{abstract}This article addresses a fundamental concern regarding the incompressible approximation of fluid motions,  one of the most widely used  approximations in fluid mechanics. Common belief is that its accuracy  is $O(\ep)$   where  $\ep$ denotes the Mach number.   In this article,  however, we prove    an   $O(\ep^2)$ accuracy for the incompressible approximation of the isentropic, compressible Euler equations thanks to several decoupling properties. At the initial time, the  velocity field and its first time derivative are of $O(1)$ size, but the boundary conditions can be as stringent as   the  solid-wall type.  The fast acoustic waves are still     $O(\ep)$ in magnitude, since the $O(\ep^2)$ error is measured in the sense of Leray projection and more physically, in time-averages.  We also show when a passive scalar is transported by the flow, it is  $O(\ep^2)$ accurate {\it pointwise in time} to use incompressible approximation for the velocity field in the transport equation.
\end{abstract}  
\date{December 31, 2013}
\section{\bf Introduction and Statement of Main Theorem}
 
All fluids are compressible, which generates acoustic waves. The restoring force   is the pressure gradient which   results from the fluid being compressed  and decompressed. The Mach number, denoted by $\ep$ in our article, is defined as the typical value of the ratio of fluid speed over sound speed. In the very subsonic regime $\ep\ll1$,   incompressible (vortical) fluid motions evolve in a slower time scale than acoustic wave propagation; then,
  incompressible approximation is often adopted so that effectively acoustic waves are filtered out. Numerous applications and theoretical studies rely on the validity of such approximation that indeed offers more convenience and simplicity than the   compressible models.  
  
  Common belief is that the incompressible approximation introduces   $O(\ep)$ errors.
 In this article,  however,  we  prove    an improved  $O(\ep^2)$ error estimate  between  the isentropic, compressible      Euler equations and its incompressible counterpart, thanks to several decoupling properties. The initial data is well-prepared in the sense that its first time derivative has   $O(1)$  spatial norms, independent of the smallness of $\ep$. In a loosely equivalent way, the velocity divergence is   only $O(\ep)$   in spatial norms and acoustic waves have only $O(\ep)$ amplitudes as well. Higher time derivatives can still grow as $\ep\to0$. The central idea of time-averaging is repeatedly used to suppress the amplitude of acoustic waves by a factor of $\ep$. Intuitively, acoustic waves oscillate fast at temporal frequencies of $O(\ep^{-1})$, and therefore averaging them in time effectively cancels out the majority of oscillations.  
 
We ought to point out that the nonlinear nature of fluid motions is bound to couple fast acoustic waves with the slower   incompressible motions.  Even when all acoustic waves are completely filtered out at the initial time, they are  instantaneously generated from slow incompressible motions.   
 In the atmosphere for example,  the ubiquitous acoustic waves are emitted  all the time,  although most are inaudible to human ears\footnote{ The sound speed is around 330 meters per second in the lower atmosphere; human's hearing range starts from 20 Hertz. Therefore, we can not hear  wave lengths longer than 17 meters in our everyday life.}.
 
To this end, time-averaging plays a crucial role to further suppress   the ``unwanted'' contribution from acoustic waves to the incompressible   dynamics.  The physical relevance of time-averaging is evident from the popularity of its generalized version, time filtering. In fact,  time filtering is necessary in dealing with observational and computational   data  when the resolution of fast acoustic waves   suffers from   a wide range of factors. To make even closer connection to applications, we will use time-averaging technique to show that, if a passive scalar is transported by a velocity field governed   by the compressible Euler equations, then it is  $O(\ep^2)$ accurate {\it pointwise in time} to replace the velocity with its incompressible counterpart(s).

Our techniques are applicable to  general bounded domains  subject to the     solid-wall  boundary condition $\vv\ccdot\nm\bigr|_{\pa\Omega}=0$. Several issues arise here: 1.  nonlinear coupling of fast    and slow   dynamics  does not decay or disperse in any strong sense; 2. Fourier analysis is not applicable; 3. straightforward   energy estimates are not convenient for proving  $\ep$-independent   estimates  in $\H^m$ norms  for the solution. The last point is related to the fact that $\vv\ccdot\nm\bigr|_{\pa\Omega}=0$ does not hold for all spatial derivatives of $\vv$ and thus   the boundary integrals (multiplied by $\ep^{-1}$) resulting from the Divergence Theorem do not vanish. These issues will be resolved by relying on time-averaging,    vorticity formulation and the simple fact $\pt^k\vv\ccdot\nm\bigr|_{\pa\Omega}=0$.

\subsection{Main results} Upon rescaling and nondimensionalization, the  isentropic, compressible Euler equations are  expressed in terms of \emph{total} density $ {\rho^\tot}$ and velocity $\vv$,\be\label{Euler:rhtot}\left\{\begin{split}\pt {\rho^\tot} +\nc( {\rho^\tot}\vv)&=0,\\
 \pt\vv +\vv\cn\vv+{1\over\ep^2}\,{\nabla \pres( {\rho^\tot})\over\rho^\tot}&=0,\end{split}\right.\ee
with   Mach number $\ep\ll1$ bringing in   acoustic waves oscillating on   fast time scales. We have had the pressure law $\pres\in C^\infty(\mR^+)$ go through rescaling and affine transformation to satisfy \be\label{pres:scale} \pres(1)=\pres'(1)=1.\ee
For instance, if $\pres(\ccdot)$ satisfies the $\gamma$-power law, then by the above   assumption,   $\pres(\rho^\tot)= (\gamma-1+(\rho^\tot)^\gamma )/\gamma$. Also, it is understood that \rev{$|\rho^\tot-1|\ll1$}, so that the pressure gradient is approximated by $\nabla\rho^\tot$ and the linearized    acoustic waves  have both phase and group velocities at order ${1/ \ep}$,  namely the rescaled  sound speed. 

 Without loss of generality, we only consider a connected (but not necessarily simply connected) compact spatial domain $\Omega\subset\mR^N$ for $N=2$ or $3$, with the ``solid-wall'' boundary condition
\[\vv\ccdot\nm\bigr|_{\pa\Omega}=0\]
 where $\nm=\nm(x)$ is the outward normal to the static, smooth boundary $\pa\Omega$. The topology of $\Omega$ will occasionally be  a concern, e.g. in Remark \ref{remark:domain}.

 The main goal of this article is to estimate, in terms of $\ep$ and initial data, the size of $(\vv-\vl)$ where $\vl$ solves the incompressible Euler equations \begin{subequations}\label{Euler:in:intro}
\begin{align}\label{Euler:in:vl}  \pt \vl+\vl\cn \vl+\nabla q&=0,\\\label{Euler:in:nc}\nc \vl&=0, \\\label{Euler:in:bc} \text{subject to\qquad} \vl\ccdot\nm\bc &=0.\\\label{Euler:in:ic}\text{Initial data $\vl_0$ satisfies }&\eqref{Euler:in:nc}, \eqref{Euler:in:bc}.\end{align}\end{subequations}
Here, the scalar $q$ is an auxiliary variable, also called pressure, that enforces the incompressible condition \eqref{Euler:in:nc}. Without such a term, the above system would be overdetermined.

The spatial $\H^m$ norm is defined as usual,\[\|f(x)\|_{\H^m}:=\Bigl(\sum_{|\beta|\le m}\intO|\pa^\beta_xf(x)|^2\,dx\Bigr)^{1/2}\]
where multi-index $\beta$ indicates   orders of derivatives taken on each spatial dimension. Let $\H^m(\Omega)$ denote the closure space of smooth functions with finite $\H^m$ norms. Of course, $\LL(\Omega)=\H^0(\Omega)$.

Before stating the Main Theorem, we clarify one technical point. The time derivates at $t=0$, denoted by $\pt^k(\rho^\tot_0,\vv_0)$, can be calculated without knowing the solution for $t>0$. This is because, by repeatedly taking time derivatives on  \eqref{Euler:rhtot}, one can inductively express $\pt(\rho^\tot,\vv),\pt^2(\rho^\tot,\vv),$ \ldots $\pt^k(\rho^\tot,\vv)$   solely   in terms of $(\rho^\tot(t,x),\vv(t,x))$  and their spatial derivatives   up to the $k$-th order evaluated at each fixed time $t$.

\begin{theorem}\label{thm:main} {\bf(Main Theorem)} Let integer $m\ge3$ and   parameter $\ep\in[0,1/2]$. Consider the compressible system \eqref{Euler:rhtot} subject to initial data  $(\rho^\tot_0-1,\vv_0)\in{\H^m(\Omega)}$.   Assume $(\rho^\tot_0,\vv_0)$ is compatible with the boundary condition, namely  $ (\pt^k\vv_0) \ccdot\nm\bigr|_{\pa\Omega}=0$ for $k<m$.

Let $\vl$ solve the incompressible   system \eqref{Euler:in:intro} subject to initial data $\vl_0=\cP\vv_0$. Here,  $\cP$,  defined in \eqref{def:cp} below, denotes the Leray projection into the  incompressible velocity subspace.

Define 
\[ \Eo:=\bigl\|({\rho^\tot-1\over\ep} ,\vv_0)\bigr\|_{\H^m},\qquad \Eto:=\bigl\|\pt({\rho^\tot\over\ep},\vv_0)\bigr\|_{\H^{m-1}}\]
Then, there exist constants $E^*,T^*, C^*$ that only depend on $m$, $\Omega$ and pressure law $\pres(\cdot)$, so that with $\Eo\le E^*/\ep$,
\be\label{es:ep2}\sup_{t\in[0,{T^*/ \Eo}]}\|\cP\vv-\vl\|_{\H^{m-3}}\le C^*\,\ep^2\, (\Eto+\Eo^2)\|\cP\vv_0\|_{\H^m}.\ee
\end{theorem}
The proof is given in the last Section \ref{sec:conclude}. 

We used the clumsy notation of ${\rho^\tot-1\over\ep}$ to state the Main Theorem, as it will be replaced  throughout the rest of this article with the density perturbation\be\label{def:rh}\rho:={{\rho^\tot}-1\over\ep}.\ee
With this  notation, \be\label{E:rho}\Eo=\|(\rho_0,\vv_0)\|_{\H^{m}},\qquad \Eto=\|\pt(\rho_0,\vv_0)\|_{\H^{m-1}},\ee
so that   the Main Theorem as well as the reformulated system \eqref{Euler} from below is invariant under the hyperbolic scaling $\rho\to c\rho,\vv\to c\vv,t\to t/c,\ep\to \ep/c$ for any constant $c>0$.

The original system \eqref{Euler:rhtot} is then reformulated in terms of  unknown pair $(\rho,\,\vv)$,
\begin{subequations}\label{Euler} \begin{align}\label{Euler:rh}\pt \rho +\nc(\rho\vv)&=-\ep^{-1}{\nc\vv }\\
\label{Euler:vv}\pt\vv+\vv\cn\vv+\hep(\rho)\nabla\rho&=- { \ep}^{-1}{\nabla \rho}\end{align}\end{subequations}
\be   \text{subject to}\quad\label{Euler:BC} \vv\ccdot\nm\bigr|_{\pa\Omega}=0.\ee 
\be\label{def:hep} \text{Here,}\quad \hep(\rho):=\left({\pres'(1+\ep\rho)\over1+\ep\rho}-1\right){1\over\ep}. \ee
Note that, by \eqref{pres:scale} and Taylor expansion,  $\hep(\rho)=(\pres''(1)-\pres'(1))\rho +  O(\ep\rho^2)$. 

In a more compact form, \[ \pt\bpm\rho\\\vv\epm+\cN\bpm\rho\\\vv\epm =-\ep^{-1}\cL \bpm\rho\\\vv\epm \]
with the nonlinear operator $\cN$ clearly defined via  \eqref{Euler}, and anti-symmetric   operator \be\label{def:cL}\cL  \bpm\rho\\\vv\epm :=\bpm\nc\vv\\\nabla\rho\epm.\ee
 For purely aesthetic reasons, we will use notations $\cL  \bpm\rho\\\vv\epm$ and $\cL ( \rho, \vv)$ interchangeably.
  
 One can easily use         \eqref{Euler}  and Sobolev inequalities   to show
\be\label{E0:E1}\Bigl| \ep\bigl\|\pt(\rho ,\vv )\bigr\|_{\H^{m-1}}-\bigl\|\cL(\rho,\vv)\bigr\|_{\H^{m-1}}\Bigr|\le C\ep \bigl\|(\rho ,\vv )\bigr\|_{\H^{m}}\bigl\|(\hep(\rho),\,\vv)\bigr\|_{\H^{m}} \quad\text{ for }m\ge3,\ee
\rev{where $\hep(\rho)\sim O(\rho)$ a la Taylor expansion.}
Thus,   having the first time derivative $\pt(\rho,\vv)\sim O(1)$ is loosely equivalent to enforcing $\cL(\rho,\vv)=(\nc\vv,\nabla\rho)\sim O(\ep)$. In other words, given $\Eo\sim O(1)$,  preparing $\Eto\sim O(1)$ is loosely equivalent to having $\rho_0$ to be $O(\ep)$ close to constant and $\vv_0$ to be $O(\ep)$ close to incompressibility.

\begin{remark}\label{remark:E0:E1}There are two bounding factors in \eqref{es:ep2} that depend on initial data. Regarding the $(\Eto+\Eo^2)$ factor, we note the compressible system \rev{ \eqref{Euler} automatically enforces $\Eto\le O(\Eo/\ep+\Eo^2)$, c.f. \eqref{E0:E1}. }Then, using  ill-prepared data that allow   $\Eto\sim O(1/\ep)$ and   acoustic waves of  $O(1)$ amplitudes, we would recover the $O(\ep \Eo)$ error estimate for ill-prepared data previously proved by B. Cheng  in \cite{Cheng:SIMA}. Regarding the other factor $\|\cP\vv_0\|_{\H^m}$, in the extreme case with purely acoustic wave or potential flow initial data,  $\cP\vv=0$ is invariantly sustained by the compressible system whereas with $\vl_0=\cP\vv_0=0$, the incompressible Euler system simply yields $\vl=0$.  Then, both  sides of \eqref{es:ep2} vanish, consistent with such well known invariance. \end{remark}

\begin{remark} The local-in-time existence and uniqueness of $\bC([0,T],\H^m(\Omega))$ solution to \eqref{Euler:rhtot} has been established in \cite{Sch:Euler}.
The compatibility condition, $(\pt^k\vv_0) \ccdot\nm\bigr|_{\pa\Omega}=0$ for $k<m$, is both  necessary and sufficient for   $\H^m$ well-posedness, although this is  not  the main  focus of this article. On the other hand, it is crucial in our study to obtain $\ep$-independent upper bounds on $\|(\rho,\vv)\|_{\H^m}$ and  $\|\pt(\rho,\vv)\|_{\H^{m-1}}$. This will be   achieved in Section \ref{sec:a:priori}.\end{remark}

\begin{remark} Estimates on $\|\cP\vv-\vl\|_{\H^{m'}}$ for   $ m'\in(m-3, m)$ are   obtained by  interpolating between $\|\cP\vv-\vl\|_{\H^{m-3}}\sim O(\ep^2)$ from \eqref{es:ep2} and     $\|\cP\vv-\vl\|_{\H^{m}}\sim O(1)$.  
\end{remark}
In the final Section \ref{sec:conclude}, we will also prove the following corollary without   relying on   Leray projection. Instead, the estimate is in terms of the physically relevant time-averages.  
\begin{corollary}\label{thm:C}Under the same hypotheses as in the Main Theorem \ref{thm:main}, there exists   constants $C^{**},C^{***}$ that only depends on $m$, $\Omega$ and pressure law $\pres(\cdot)$, so that for all  times $T\in[0,T^*/\Eo]$,
\be\label{es:C} \ep\bigl\|\vv-\vl\bigr\|_{ \H^{m-3}}(T)+\bigl\|\intT \vv-\vl \bigr\|_{ \H^{m-3}}\le \,C^{**}\,\ep^2\,(\Eto+\Eo^2) \ee

Moreover, if two scalars $\th(t,x)$, $\thl(t,x)$ are transported by $\vv,\vl$ respectively
\[\begin{aligned}\pt\th+\vv\cn\th&=0,\\\pt\thl+\vl\cn\thl&=0,\end{aligned}\]
subject to the same initial data $\th_0=\thl_0\in{\H^m(\Omega)} $, then
\[\sup_{[0,{T^*/ \Eo}]}\|\th-\thl\|_{\H^{m-3}}\le \,C^{***}\,\ep^2\,(\Eto+\Eo^2)\, \|\th_0\|_{\H^m}.\]
Apparently, by \eqref{es:ep2},   the same $O(\ep^2)$ accuracy holds true if we approximate the velocity by $\cP\vv$.
\end{corollary}
Note in \eqref{es:C} the gained $\ep$ factor thanks to time averaging. Also, it is  $O(\ep^2)$ accurate {\it pointwise in time} to use both $\vl$ and $\cP\vv$ approximations for $\vv$ in the transport equation.

\subsection{Literature} There have been numerous results regarding the  singular limits  of compressible Euler equations and other
fluid equations  in various settings, \rev{but literature on  the $O(\ep^2)$ accuracy is limited. To our best knowledge, such results are all represented in the form of   $O(\ep),O(\ep^2)\ldots$ corrections to the incompressible approximation (e.g. \cite[Theorem 3.2]{Sch:limits}, \cite[Theorem 3]{Gallagher} for periodic domains under various ``small divisor'' conditions)}. Our result here confirms that the $O(\ep)$ correction is in fact zero for Euler equations in general spatial domains, \rev{at least for well-prepared data}. We point to two survey papers for some comprehensive lists of references:   \cite{Sch:survey} with emphases on hyperbolic partial differential equations (PDEs) and homogenization in space-time;   \cite{Mas:survey} with emphases on viscous fluids and weak solutions. To mention only a few earliest works in terms of well-prepared data, we refer to  \cite{Ebin:earlier, Ebin:Euler,Kreiss:multiscale,BDV,Kl:Maj,Tad:multiscale}.  In a closely related paper \cite{Kreiss:derivative},   the bounded derivative method is applied to numerical schemes from geophysical applications. 
Well-prepared conditions on initial data were later removed for problems in the whole space (\cite{Ukai:Rn}), in an exterior domain (\cite{Isozaki:exterior,Isozaki:wave}) and in a torus (\cite{Sch:limits}). These arguments more or less rely on use of Fourier analysis and/or dispersive nature of the underlying wave equations. 

Singular limit problems in a bounded spatial domain, on the other hand, remain much less studied.  Even the well-posedness of hyperbolic PDEs in a bounded domain can be challenging due to  the (possible)   characteristic boundary. For example, Rauch elaborated in   \cite{Rauch:char} that, in such settings,   only estimates along tangential directions   are available near the boundary. Nevertheless,   Schochet \cite{Sch:Euler} proved the   low-Mach-number limit with solid-wall boundary condition and well-prepared initial data, but without convergence rates.   In  \rev{\cite{Secchi}}, Secchi proved the strong convergence of $\cP \vv $   for 3D Euler equations with ill-prepared initial data, again without   convergence rates. Very recently, B. Cheng proved $O(\ep)$ convergence rate for ill-prepared data in \cite{Cheng:SIMA}. The time-averaging technique used there inspired this current study; also see Cheng \& Mahalov \cite{Ch:Ma} for time-averaging applied to geophysical models on a sphere.
   
\subsection{\bf \rev{Slow dynamics and vortical dynamics}}
 
Later in the article, we will apply the Leray projection to   the compressible system, which effectively annihilates $\ep^{-1}\cL$. This gives a decomposition of the solution space into slow and fast subspaces, and correspondingly a decomposition of the compressible system into a slow one governing the  incompressible motions  and a fast one   governing the rapidly oscillating   acoustic waves.
 
The slow dynamics is very closely related to the vorticity equations.  Apply $\nt$ to \eqref{Euler:vv} 
so that the cancellation property $\nt\nabla=0$  yields the   equation for vorticity $\om:=\nabla\times\vv$ \begin{subequations}\label{vorticity:Euler}
\begin{align}\label{vorticity:2DEuler} \pt\om  +\vv\cn\om +(\nc\vv)\om =&0\qquad\text{in 2D, }\\
\label{vorticity:3DEuler} \text{and }\quad\pt\om -\om\cn\vv+\vv\cn\om +(\nc\vv)\om=&0 \qquad\text{in 3D}\end{align}
\end{subequations}without $O(\ep^{-1})$ terms contributing to $\pt\om$. Thus, vorticity $\om$ evolves on a slow time scale. We have to alert that the   vortical dynamics may not contain all the information of the slow dynamics if the spatial domain is not contractible. Take Remark \ref{remark:domain} for example. A point vortex restricted in an annulus domain forms a steady solution of  incompressible  Euler equations; together with  a suitable axisymmetric   profile of $\rho$, it also solves the compressible system. However, both   its  vorticity and divergence     are identically zero;   therefore the    vorticity equation, even combined with the divergence equation, does not retain all dynamical information required to solve for $\vv$.

 Nevertheless, the vorticity equation is   widely used in practice as it has the simple structure of a transport equation which turns out to be crucial for  some estimate proofs later on.
 
The rest of this article is organized as follows. In Section 2, we define the Leray projection, prove its properties using elliptic PDE theory and use it to extract the slow dynamics from  the compressible system.   Section \ref{sec:ave} contains probably the most novelty. It explains how to use the time-averaging technique to obtain pointwise-in-time error estimates. A decoupling property particular to  the compressible Euler system will  allow us to gain an extra $\ep$ factor, provided  the  data are well-prepared. Next, without concerns for boundary, the reader can skip   Section \ref{sec:a:priori}. Here, we use mixed norms to obtain $\ep$-independent bounds on  the sizes of the solution and its first time derivative.   The methods used here are partially similar to those of \cite{Sch:general}, but we work with both ill-prepared and well-prepared data. The     final Section \ref{sec:conclude} completes   proofs of the Main Theorem \ref{thm:main}, Corollary \ref{thm:C} and makes  further comments.

\begin{remark}\label{re:Sobolev}The Sobolev inequalities used throughout this article can   be summarized as follows,   which are proven in the Appendix. Given  functions $f_1( x), f_2( x),\ldots,f_j( x)$ over a two or three dimensional compact domain $\Omega$,    
 we have estimate  \be\label{gen:Sob}\Bigl\|\prod_{i=1}^j\px^{\beta_i}f_i \Bigr\|_{\LL}\lsm \prod_{i=1}^j\Bigl\|f_i\Bigr\|_{\H^k} \ee
 if   one of the following conditions hold
 \be\label{cond:k}\left\{\begin{aligned}&|\beta_1+...+\beta_j|\le k,\;\;\text{ and }\;\; k\ge2,\;\;\text{ or }\\
 &|\beta_1+...+\beta_j|= k+1,\;\; \max\{|\beta_1|,\ldots,|\beta_j|\}\le k,\;\;\text{ and }\;\; k\ge3.\end{aligned}\right.\ee
Here and below, the ``similarly less than'' notation $a\lsm b$ is understood as
``$a\le Cb$  for a constant $C$ solely depending on   $\Omega$,  the pressure law $\pres(\ccdot)$ and the Sobolev spaces involved''.\end{remark}

\section{\bf Projections onto Slow and Fast Subspaces}\label{sec:proj}

Define $\sX$ to be the space of incompressible velocity fields subject to solid-wall boundary condition, 
\rev{
\be\label{sX:testing}\begin{aligned}\sX&= \LL\text{ closure of }\left\{\vv\in \bC^1(\overline{\Omega})\,\Big|\,\intO \vv\cn f=0\text{ for any }f\in \H^1(\overline\Omega)\right\}\end{aligned}\ee
}
 
Define $\cP$ as the $\LL$-orthogonal projection onto  $\sX $ so that, for any $\vv,\vv'\in \LL(\Omega)$, \begin{subequations}\label{def:cp} \begin{align}\label{def:cp:1} \cP^2\vv=\cP\vv \in \sX,&\\
\label{def:cp:2}\intO(\vv-\cP\vv)\ccdot(\cP\vv')&=0.\end{align}\end{subequations}
In fact, $\cP$ is the classical Leray projection subject to solid-wall boundary condition. Then, define 
\[\cQ:=1-\cP.\]

 These projections can  be characterized conveniently by an elliptic PDE as follows.
\begin{proposition}\label{thm:cQ} For any $\vv\in \H^k(\Omega)$ with $k\ge1$, we have  \be\label{def:cQ}
{\cQ\vv=\nabla\phi} \in \H^k(\Omega)\text{ \;\;where\;\; }\phi\mbox{ solves }\left\{\begin{split}\Delta\phi&=\nc\vv&\text{in }\Omega, \\\nabla\phi\cdot\nm&=\vv\ccdot\nm&\text{in } {\pa\Omega}.\end{split}\right.\ee
  Also, $\phi$ is unique up to  an added constant  and thus $\cQ\vv$ is unique.
 \end{proposition}
 Here and below, we always assume $k\ge1$ whenever the trace of an $\H^k(\Omega)$ function is involved.
\begin{proof}
The solvability   of \eqref{def:cQ}   follow  from standard elliptic PDE theory (e.g. \cite{Taylor:I}, Ch. 5, Prop. 7.7). It suffices to verify that $(1-\cQ)$ with $\cQ$ given in \eqref{def:cQ} equals $\cP$ which is uniquely defined in \eqref{def:cp:1}, \eqref{def:cp:2}.

Obviously, \eqref{def:cQ} implies $ (1-\cQ) \vv\in \sX$  for any $\vv\in \H^k(\Omega)$.
Then, for  any  $\vv^\text{inc}\in\sX\bigcap \H^k(\Omega)$, the uniqueness   of   $\phi$ in \eqref{def:cQ}, up to an added constant, implies  that  $\cQ\vv^\text{inc}=0$. In short,  $(1-\cQ)$ is a projection so that, for any $\vv\in \H^k$,   $(1-\cQ)^2 \vv=(1-\cQ)\vv\in\sX$.

It remains to show \eqref{def:cp:2} i.e. that $\cQ \vv$ is $\LL$-orthogonal to any  $\vv^\text{inc}\in\sX$. Since $\sX\bigcap \bC^\infty(\overline{\Omega})$   is   dense in $\sX $ in the $\LL$ topology, it suffices to have     $\vv^\text{inc} $ as a  smooth element of $\sX$. Then, \[\intO \vv^\text{inc}\ccdot(\cQ\vv)\,dx=\intO\vv^\text{inc}\ccdot\nabla\phi\,dx\stackrel{ {(a)}}{=}\intO\nc(\vv^\text{inc}\phi)\,dx\stackrel{ {(b)}}{=}\int_{\pa\Omega}\nm\ccdot \vv^\text{inc}\phi\,ds\stackrel{ {(c)}}{=}0,\] where $(a),(c)$ are due to the definition of $\vv^\text{inc}$ and $(b)$ due to the Divergence Theorem.
\end{proof}
This proposition shows that $\cQ\vv$  is always a perfect gradient and therefore its curl vanishes.  
\begin{align}\label{cKcq}\nt(\cQ\vv)=0,&\qquad \nt(\cP\vv)=\nt \vv .\\
\intertext{Compare them to the definitional facts,}
\nonumber \nc(\cP\vv)=0,&\qquad\nc(\cQ\vv)=\nc\vv.\end{align}
This is to say, $\cP\vv$ contains all the information of the vorticity (but not necessarily vice versa!) and  $\cQ\vv$ contains all the information of the divergence.

 From here on,  we will interchangeably use   $\vp$ for $\cP\vv$ and  $\vq$ for $\cQ\vv$.

\subsection{Boundedness of projections and elliptic estimates}

 Operators $\cL,\nabla\times, \cP,\cQ$ are all elliptic operators with nontrivial kernels, and we will employ elliptic estimates with boundary conditions to estimate them. The papers of Agmon, Douglis and Nirenberg \cite{elliptic:1}, \cite{elliptic:2} establish a ``Complementing Boundary Condition'' that is necessary and sufficient for the solution operator of an $s$-th order elliptic PDE system to be $\bC^k\to\bC^{k+s}$ and $\H^k\to \H^{k+s}$. To treat the Euler equations (e.g. \cite{BB:Euler}),  only a particular case is used: for any velocity field $\vv\in\H^1(\Omega)$,
\be\label{elliptic:bc}\|\vv\|_{\H^k}\lsm\|\nc\vv\|_{\H^{k-1}}+\|\nt\vv\|_{\H^{k-1}}+\|\vv\|_{\LL},\quad\text{if }\vv\ccdot\nm\bc=0 \text{ and }k\ge1.\ee 

 \begin{remark}\label{remark:domain}We alert that the $\|\vv\|_{\LL}$ term above  may not be dropped if, for example, $\Omega$ is not contractible. Consider an annulus domain $ \Omega=\{(x,y)\,\big|\,1<x^2+y^2<2\}$ and a point vortex $\vv=\nabla^\perp\ln\bigl| x^2+y^2\bigr|$.   Then, $\nabla\ccdot\vv=\nabla\!\times\!\vv=0$ and $\vv\ccdot\nm\bc=0$, but $\vv\ne0$.
 \end{remark}

Now,   set $\vv=\vp$ in  \eqref{elliptic:bc} and use the   facts that $\nc\vp=0$, $\vp\ccdot\nm\bc=0$ to  obtain, without   boundary condition on $\vv$,
\be\label{es:cp:L2}\|\vp\|_{\H^k}\lsm \|\nt\vp\|_{\H^{k-1}}+\|\vp\|_{\LL}\quad\text{ for }k\ge1.\ee
This gives a bound on the high norms of $\vp$ using the high norms of $\nt\vp=\nt\vv$ (by \eqref{cKcq}) and the $\LL$ norm of $\vp$. 

Therefore, $\cP$, $\cQ$ are bounded operators in $\H^k(\Omega)$ regardless of boundary condition,
 \be\label{es:cp:cq} \|\vv\|_{\H^k}\le\|\cP\vv\|_{\H^k}+\|\cQ\vv \|_{\H^k}\lsm \|\vv\|_{\H^k}\quad\text{ for }k\ge0.\ee
 The first inequality is apparently due to  $\cP+\cQ=1$. Also, for the second inequality,
the case of $k=0$ is due to the definition of $\cP$ and the Pythagorean Theorem. 

Similar to \eqref{es:cp:L2}, we can bound the norms of $\vq$ using norms of $\nc\vq=\nc\vv$. In fact, set $\vv=\vq$ in \eqref{elliptic:bc},      use \eqref{cKcq} and the fact $\vv\ccdot\nm\bc=\vq\ccdot\nm\bc$ to obtain
\be\label{es:cQ:L2}\|\vq\|_{\H^k}\lsm\|\nc\vq \|_{\H^{k-1}}+\|\vq\|_{\LL}\quad\text{ if }\quad \vv\ccdot\nm\bc=0\ee 

A notable feature of the above inequality is, unlike in Remark \ref{remark:domain}, the $\LL$ norm term above can be dropped regardless of topology of the spatial domain.  
\begin{proposition}\label{prop:cQ}Let $k\ge1$. For any $\vv\in \H^k(\Omega)$,\[\| \vq\|_{\H^k}\lsm\|\nc\vq\|_{\H^{k-1}}=\|\nc\vv\|_{\H^{k-1}}\quad\text{ if }\quad \vv\ccdot\nm\bc=0.\]\end{proposition}
\begin{proof} 
Having \eqref{es:cQ:L2} established,  it suffices to show $\|\vq\|_{\LL}\le C_\Omega\|\nc\vv\|_{\LL}$.  

Take any smooth test velocity field $\vv'$. By   Proposition \ref{thm:cQ},   $\vv'=\nabla\psi+\cP\vv'$. Then,   the orthogonality of $\cP$, $\cQ$ implies  $\intO\vq\ccdot\vv'=\intO\vv\ccdot\nabla\psi.$ 
    Apply  the Divergence Theorem and $\vv\ccdot\nm\bc=0$ on the RHS to get $\intO\vq\ccdot\vv'= -\intO(\nc\vv )\psi.$ Now, since $\Omega$ is compact,  we can  set $\intO\psi=0$, and apply  the H\"older and Poincar\'e inequalities, 
\[\intO\vq\ccdot\vv' \le C_\Omega\|\nc\vv\|_{\LL}\|\nabla\psi\|_\LL\le C_\Omega\|\nc\vv\|_{\LL}\|\vv'\|_\LL,\]
where the last inequality is due to     $\|\nabla\psi\|_\LL=\|\cQ\vv'\|_\LL $ and the Pythagorean Theorem. This shows $\|\vq\|_{\LL}\le C_\Omega\|\nc\vv\|_{\LL}$.  
  \end{proof}
 
  \subsection{Slow and Fast dynamics}

Next, we want to extract the slow dynamics from \eqref{Euler} in the form of an evolutionary system that is free of  $O(\ep^{-1})$ time derivative. One could apply $\nt$ to cancel $\ep^{-1}\cL$ and get the vorticity equations \eqref{vorticity:2DEuler} or \eqref{vorticity:3DEuler}, but the comments thereafter suggests that the vortical dynamics does not necessarily retain all the information of the slow dynamics. For such a reason, we will instead apply  $\cP$ on \eqref{Euler:vv}. The intuition is, if we define $\cP^\sharp\bpm\rho\\\vv\epm:=\bpm0\\\cP\vv\epm$ to make $\cL\cP^\sharp=0$ and to make $\cL$, $\cP^\sharp$   (skew-)symmetric, then    hopefully they commute   $\cP^\sharp\cL=0$ and therefore applying $\cP^\sharp$ to \eqref{Euler} will  eliminate the $\ep^{-1}\cL$ term.  This is easily proved using adjoint operators if $\pa\Omega=\emptyset$, and can still be established in general provided the boundary conditions are taken care of. 
  \begin{proposition}\label{thm:cPcL}
  For any scalar $\rho\in \H^1(\Omega)$,\[   \quad\cP(\nabla \rho)=0.\]
  \end{proposition}
  \begin{proof}  By orthogonality of $\cP$, $\cQ$, we have
  \[\intO\cP(\nabla\rho)\ccdot \cP(\nabla\rho)=\intO\cP(\nabla\rho)\cn\rho\]
 which is zero due to $\cP(\nabla\rho)\in\sX$ satisfying \eqref{sX:testing}.\end{proof}
Thus, we apply $\cP$ on \eqref{Euler:vv} to obtain the slow dynamics,
\be\label{UP:B:fast-fast}  -\pt \vp= \cP\bigl(\vv\cn\vv\bigr)= \cP\bigl(\vp\cn\vp\bigr)+ \cP\bigl(\vp\cn\vq+\vq\cn\vp\bigr)+ \cP\bigl(\vq\cn\vq\bigr) \ee
where we also used the fact that $\hep(\rho)\nabla\rho$ is a perfect gradient and therefore   is annihilated by $\cP$ according to the   above proposition.

On the other hand,   apply  $\cQ$ to \eqref{Euler:vv} and keep  \eqref{Euler:rh}   to  obtain the fast dynamics
\begin{subequations}\label{UQ} \begin{align}\label{UQ:rh}\pt \rho +\nc(\rho\vv)&=-\ep^{-1}{\nc\vv }\\
\label{UQ:vv}\pt\vq+\cQ(\vv\cn\vv)+\hep(\rho)\nabla\rho&=- \ep^{-1}{\nabla \rho }\end{align}\end{subequations}
This way, the original system is decomposed into   \eqref{UQ} governing the fast variables $\rho, \vq$  with   $O(\ep^{-1})$ coefficients, and   \eqref{UP:B:fast-fast} governing the slow variable $\vp $ whose first    time derivative is   $O(1)$. Note  density component is identically zero in the slow variable.

A key {\bf decoupling} property   is that the ``fast-fast'' product in \eqref{UP:B:fast-fast}  vanishes completely.
\begin{lemma}\label{ff:lemma}For any $\vq\in \H^2 ( {\Omega})$,
\[ \cP\bigl(\vq\cn\vq\bigr)=0.\]
\end{lemma}
\begin{proof}
By Proposition \ref{thm:cQ},  there exists a   scalar function $\phi$ so that  $\vq=\nabla \phi.$ Then, $\vq\cn\vq=(\nabla \phi)\cn(\nabla\phi)={1\over2}\nabla|\nabla\phi|^2$. So,  by Proposition \ref{thm:cPcL}, $ \cP\bigl(\vq\cn\vq\bigr)$ vanishes.
\end{proof}
Therefore,  we rewrite \eqref{UP:B:fast-fast}  as
\be  \label{UP:B} -\pt \vp = \cP\bigl(\vp\cn\vp\bigr)+ \cP\rB{ \vp, \vq},\ee
where bilinear operator  $$\rB{\vv_1,\vv_2}:=\vv_1\cn\vv_2+\vv_2\cn\vv_1.$$
If we set $\vq\equiv0$ by brutal force, \eqref{UP:B} would be {\it formally} reduced to  the incompressible Euler equations.
\begin{proposition}\label{thm:equiv}
Consider a  velocity field  $\vl\in \bC([0,T^*]\times\H^3(\Omega))$. Then, $\vl$ solves the {\it actual} { incompressible} Euler equations \eqref{Euler:in:vl}, \eqref{Euler:in:nc}, \eqref{Euler:in:bc} if and only if it solves
\be  \label{Euler:in:cp} -\pt \vl = \cP\bigl(\vl\cn\vl\bigr) \ee
with  the same   initial data $\vl_0$ satisfying \eqref{Euler:in:ic}.\end{proposition}
\begin{proof} `` Only if ''. Assume $\vl$ solves \eqref{Euler:in:vl}, \eqref{Euler:in:nc}, \eqref{Euler:in:bc}. Apply $\cP$ on \eqref{Euler:in:vl}. On the LHS, because $\pt\vl$ also satisfies \eqref{Euler:in:nc}, \eqref{Euler:in:bc},     we have $\cP(\pt\vl)=\pt\vl$ by the definition of $\cP$. Also, we have $\cP(\nabla q)=0$ by Proposition \ref{thm:cPcL}.  Therefore, applying $\cP$ on \eqref{Euler:in:vl} gives us exactly \eqref{Euler:in:cp}.

`` If ''. Assume $\vl$ solves \eqref{Euler:in:cp}, which can be recast as
\[ \pt \vl +\vl\cn\vl -  \cQ\bigl(\vl\cn\vl\bigr)=0  \]
By Proposition \ref{thm:cQ}, there exists a scalar $\phi$ so that  $\cQ\bigl(\vl\cn\vl\bigr)=\nabla\phi$ and therefore the above equation is of the same form as \eqref{Euler:in:vl} with pressure $q=-\phi $.  Next,  since \eqref{Euler:in:ic} ensures $\nc\vl_0=0$ and taking divergence on \eqref{Euler:in:cp} gives $\pt(\nc\vl)=0$, we have $\nc\vl=0$, i.e. \eqref{Euler:in:nc} satisfied for all   $t\in [0,T^*]$. Finally, restrict \eqref{Euler:in:cp} on $\pa\Omega$, take its dot product with $\nm$ and use the definition of $\cP$ to obtain
\(\pt(\vl\ccdot\nm)=0\)  on \(\pa\Omega.\)
Since \eqref{Euler:in:ic} ensures $\vl_0\ccdot\nm\bc=0$, we have $\vl\ccdot\nm\bc=0$, i.e. \eqref{Euler:in:bc} validated for all   $t\in [0,T^*]$.
\end{proof}
Problem is, in order to estimate the difference of \eqref{UP:B} and \eqref{Euler:in:cp}, how can we bound the ``slow-fast'' term  $\cP\rB{\vp, \vq}$ by $O(\ep^2)$? Because of nonlinearity, the slow subspace $\ker\cL$ is not invariant under the coupled slow-fast dynamics. This means, even with initial data $\vq_0=0$ and $\vp_0\sim O(1)$,  nonlinear coupling can lead to $\vq\sim O(\ep)$   in later times. 

To this end, we bring to focus the key idea of this article:   a generic compressible solution $\vv$,  without being $O(\ep^2)$ pointwise in time,   can still be $O(\ep^2)$  in terms of its {\it time-averages} as long as $\vq_0\sim O(\ep)$ initially. Such estimate in turn will suffice to make $(\vp-\vl)\sim O(\ep^2)$ pointwise in time. This is the subject of the next   section. 

\section{\bf Pointwise-in-time Error Estimates Using Time-averages}\label{sec:ave}
In this section,   we demonstrate in Lemma \ref{thm:ave:cb} and Theorem \ref{thm:ave:es} the crucial role of time-averages in estimating    $(\vp-\vl)$ {\bf pointwise in time}.  
For brevity, throughout this section, we   assume solutions $\rho,\vv,\vl\in \bC([0,T], \H^m(\Omega))$ for integer $m\ge3$.  

Define the time-averaging (indeed, integrating) operator
\[\ave{\vv}(T,x):=\int_0^T\vv(t,x)\,dt\]
 
First,   estimate   the  slow-fast product $\rB{\vp,\vq }$ of \eqref{UP:B}, which is  the extra term compared to the incompressible system  \eqref{Euler:in:cp}. By the product rule, $ { \rB{\vp,\vq }}= \pt {\rB{\vp,\ave{\vq}}}-{\rB{\pt \vp,\ave{\vq}}}$, so we  apply time averaging,
\[ \begin{aligned}\ave{ \rB{\vp,\vq }}(t)= &\,{\rB{\vp,\ave{\vq}}}\Big|_0^t-\int_0^t{\rB{\pt \vp,\ave{\vq}}}\\=&\,{\rB{\vp,\ave{\vq}}}\Big|_0^t+\int_0^t{\rB{\cP({\vp\cn\vp})+\cP\rB{\vp,\vq},\,\ave{\vq}}}\end{aligned}\]where    $\pt\vp$ was replaced via the slow dynamics \eqref{UP:B}.
Apply the bounds  of $\cP$,   $\cQ$ as in \eqref{es:cp:cq} and the Sobolev inequalities \eqref{gen:Sob} to obtain,\be\label{es:ave:cb}  \eesup\|\ave{ \rB{\vp,\vq }}\|_{  \H^{m-2} } \lsm \esup\Bigl\{\|\ave{\vq}\|_{\H^{m}}  \|\vp\|_{\H^{m }}(1+T  \|\vv\|_{\H^m}) \Bigr\}\ee  
Here,      the   $\|\ave{\vq}\|_{\H^{m}}$ factor measures the size of time-averaged fast variable, for which we will prove a crucial   $O(\ep^2)$ upper bound. In fact,  by Proposition \ref{prop:cQ}, 
 \[\|\ave{\vq}\|_{\H^{m}}\lsm \|\ave{\nc\vv}\|_{\H^{m-1}}\lsm\|\nabla(\ave{\nc\vv})\|_{\H^{m-2}}\]
 where the last estimate is due to the Poincar\'{e} inequality and   the zero spatial mean of $\nc\vv $. Now,  replace    $\nabla(\ave{\nc\vv})$ on the RHS using the  continuity equation \eqref{UQ:rh},
 \be\label{es:ave:vq}\|\ave{\vq}\|_{\H^{m}}(T)\lsm \ep\left\|\nabla(\rho(T,\ccdot)-\rho_0)+\nabla\intT{ \nc(\rho\vv)} \right\|_{\H^{m-2}}.\ee 
 
 Already, we have gained an $\ep$ factor   in the bound of $\ave{\vq}$. But this is not enough for $O(\ep^2)$.  A key {\bf decoupling property} here    is that the quadratic terms   in the RHS   contain    no ``slow-slow'' product which would have made $O(\|\vp\|^2)\sim O(1)$ contribution. Instead, everything in the RHS of \eqref{es:ave:vq}  has a factor    from the fast variables $\rho,\vq$, or more precisely   $(\nabla\ccdot\vq, \nabla\rho)=\cL(\rho, \vv)$.  Thus, combine \eqref{es:ave:vq} and Sobolev inequalities  \eqref{gen:Sob}  to get \be\label{ave:vq:cL}\|\ave{\vq}\|_{\H^{m}}(T)\lsm\ep\, \eesup\Bigl\{\bigl\|\cL(\rho, \vv)\bigr\|_{\H^{m-1}} (1+T \|(\rho,\vv)\|_{\H^{m}})\Bigr\}.\ee
 Plug it into  \eqref{es:ave:cb}  to   prove the following lemma.
 \begin{lemma}\label{thm:ave:cb}Let integer $m\ge3$. Suppose the compressible Euler equations \eqref{Euler} admit   solution $(\rho,\vv)\in \bC([0,T],\H^m(\Omega))$. Then, there exists a constant $C$ solely depending on $m$, $\Omega$, so that
 \[\begin{split}\eesup\|\ave{ \rB{\vp,\vq }}\|_{  \H^{m-2} }\le C\,\ep\,  & \eesup\Bigl\{\bigl\|\cL(\rho, \vv)\bigr\|_{\H^{m-1}} \|\vp\|_{\H^{m }}(1+T  \|(\rho,\vv)\|_{\H^m})^2\Bigr\}\end{split}\]
\end{lemma}

Meanwhile,   estimate \eqref{E0:E1} relates $\cL(\rho, \vv)\sim O(\ep)$ to $\pt(\rho,\vv)\sim O(1)$.  Thus, we will show in Section \ref{sec:a:priori}, by preparing initial data so that    $\bigl\|\pt(\rho,\vv)\bigr\|_{\H^{m-1}}\sim O(1)$  at $t=0$, it will remain $O(1)$  for finite times,   making   $\cL(\rho, \vv)\sim O(\ep)$ hence making the RHS in the above lemma   $O(\ep^2)$.  
 
Having bounded the time-average of the  ``extra term''  $ \rB{\vp,\vq }$ in the compressible system \eqref{UP:B}, we move on to  show how it helps us to  estimate  $(\vp-\vl)$.  Before the main result in Theorem \ref{thm:ave:es}, we prove
  following technical lemma showing that  error propagation in  a bilinear  time-dependent system   heavily relies   on the time-averages of its coefficients.

\begin{lemma}\label{thm:linear:es}
Let $\rb(\cdot,\cdot)$ be a bilinear operator (made precise below).
Consider time-dependent    systems 
\be\label{linear:bilinear}\pt\su_i+\rb{(\sv_i,\,\su_i)}=\si_i,\qquad\text{  }i=1,2,\qquad\text{subject to the same initial data}\ee
with $\sv_i\in\bC([0,T],\scrB)$, $\su_i\in\bC([0,T],\scrB')$ and $\ave{\si_1-\si_2}\in\bC([0,T],\scrB'')$ for some Banach spaces $\scrB,\scrB',\scrB''$.

Let $\scrH\supset\scrB',\scrB''$ be a (real) Hilbert space so that  $\rb(\cdot,\cdot)$ is bounded as     $\scrB\times\scrB'\mapsto\scrH$, and also bounded as   $\scrB\times\scrB''\mapsto\scrH$. Furthermore, assume 
\be\label{es:b:3}\langle u,\rb(v,u) \rangle_\scrH\le C\|v\|_\scrB\|u\|_\scrH^2,\quad\text{ for }v\in\scrB,\;u\in\scrB'\bigcap\scrB'',\ee
with some constant $C$.
Then,  with $M:=\eesup\bigl\{\| \sv_1\|_{\scrB}  \bigr\}$, we have\[ \eesup \|\su_1-\su_2 \|_{ \scrH } \le \eesup\|\ave{\si_1-\si_2}\|_{\scrH}+{e^{CMT}-1\over CM}\eesup\Bigl\{ \big\|\rb(\sv_1,\ave{\si_1-\si_2})+\rb(\sv_1-\sv_2,\,\su_2)\big\|_{\scrH} \Bigr\} \]
\end{lemma}

The significance of this result is that $|\su_1-\su_2|$ measured   {pointwise in time}  is affected by   $\si_i$ only via   time-average    $\ave{(\si_1-\si_2)}$, not the  pointwise-in-time values of  $|\si_1-\si_2|$. 

Similar effect comes from $\ave{(\sv_1-\sv_2)}$ as well,
but we will neither prove nor use it in this article. 
\begin{proof}  (Lemma \ref{thm:linear:es}).    Set $i=1,2$ in \eqref{linear:bilinear} and subtract them to get
\[\pt(\su_1-\su_2)+\rb{(\sv_1 ,\,\su_1-\su_2)}+\rb{(\sv_1-\sv_2,\,\su_2)}= \si_1-\si_2 \]
Then, define the time integral of the right hand side
\[\xi  :=\ave{(\si_1-\si_2)}, \]
and  replace the RHS of the previous equation with $\pt\xi$,  recasting it into
\[\pt(\su_1-\su_2-\xi)+\rb(\sv_1, \su_1-\su_2-\xi)=-\rb(\sv_1,\xi)-\rb(\sv_1-\sv_2,\,\su_2).\]
The boundedness hypotheses on $\rb(\cdot,\cdot)$ guarantee every bilinear term above is in  $\bC([0,T],\scrH)$ and therefore so is $\pt(\su_1-\su_2-\xi)$. This allows us to 
take the $\scrH $ inner product of this equation with $(\su_1-\su_2-\xi)$. Then, apply  \eqref{es:b:3} to get
\[ {d\over dt}\|\su_1-\su_2-\xi\|_{  \scrH } \le C\|\sv_1\|_{ \scrB}\|\su_1-\su_2-\xi\|_{ \scrH}+\big\|\rb(\sv_1,\ave{\si_1-\si_2})+\rb(\sv_1-\sv_2,\,\su_2)\big\|_{\scrH} .\]
 
Finally, for fixed $T$, relax $ \| \sv_1\|_{\scrB} $    to $M$
 and also relax the last   term to its maximum over $[0,T]$ to arrive at a   differential inequality with constant coefficients. This is easily solved to confirm the desired conclusion.
\end{proof}

 We also state a convenient fact   due to   $\H^2\subset\sL^\infty$ on two or three dimensional domain,  
 \be\label{b:H:3}\text{for }k\ge 0,\quad \|fg\|_{\H^{k }}\lsm\|f\|_{\H^{k+2}}\|g\|_{\H^{k }}.\ee
 
\begin{theorem}\label{thm:ave:es} {\bf(Time-averaging estimates) } Consider the  incompressible Euler equations \eqref{Euler:in:cp} and the slow dynamics \eqref{UP:B},  i.e.
\begin{align}  \label{Euler:in:cp:lemma}  \pt \vl+ \cP\bigl(\vl\cn\vl\bigr)&=0,\\ 
   \label{UP:B:lemma}  \pt \vp+ \cP\bigl(\vp\cn\vp\bigr)&=- \cP\rB{ \vp, \vq},\end{align}
both of which are subject to  the solid-wall boundary condition and   the same initial data $ \vl_0=\vp_0$. 
 Suppose    $(\vl,\rho,\vv =\vp+\vq)\in \bC([0,T^\flat ],\H^m(\Omega))$ for $m\ge3$.  
Then, there exists   constants $D_{1},D_2$ only dependent on $m$, $\Omega$ so that,\[\eesup\|\vl-\vp\|_{ \H^{m-3} } \le \ep\,D_{1}\,\eesup\bigl\{\bigl\|\cL(\rho, \vv)\bigr\|_{\H^{m-1}} \|\vp\|_{\H^m}\bigr\}\]for $T\in[0,T^\flat ]\bigcap\Bigl[0,\,  \displaystyle{D_2/\esssup_{[0,T^\flat]} \| (\rho,\vv,\vl)    \|_{\H^{m}}}\Bigr]$.
\end{theorem}
 \begin{proof}Recall we defined $\sX$   in \eqref{sX:testing} as the space of incompressible velocity fields subject to solid-wall boundary condition.
To fit the notations of Lemma \ref{thm:linear:es},   set $$\begin{aligned}u_1=v_1=\vl,\quad&\si_1=0,\\u_2=v_2=\vp,\quad&\si_2=- \cP\rB{ \vp, \vq},\\
 \scrB=\scrB'=\H^m\bigcap\sX,\quad&\scrB''=\H^{m-2}\bigcap\sX\end{aligned}$$ and in particular set $$\scrH=\H^{m-3}\bigcap\sX.$$   We then endow $\scrH$ with the following inner product
\be\label{def:scrH}\langle\vv,\vv'\rangle_\scrH:=\begin{cases}\langle\vv,\vv'\rangle_\LL,&\text{ if }m=3,\\
\langle\vv,\vv'\rangle_\LL+\langle\nt\vv,\nt\vv'\rangle_{\H^{m-4}},&\text{ if }m\ge4.\end{cases}\ee
The induced $\scrH$ norm, by the virtue of \eqref{es:cp:L2}, is then equivalent to the $\H^{m-3}$ norm
\be\label{H:H:equiv}\|\vv\|_\scrH\lsm\|\vv\|_{\H^{m-3}}\lsm\|\vv\|_\scrH.\ee
By \eqref{es:cp:cq}, we see $\cP$ is a bounded operator over $\scrH$ as well as bounded over $\H^{m-3}$.

Set the bilinear operator  
\be\label{def:b:2:3}\rb(\vv,\vv')=\cP(\vv\cn\vv'),\ee
so that, by    \eqref{H:H:equiv}
  and  \eqref{b:H:3} ($k=m-3$), all the boundedness assumptions on $\rb(\cdot,\cdot)$ in Lemma \ref{thm:linear:es} are satisfied. This  leaves us only \eqref{es:b:3} to validate, i.e we need to show, with $m\ge3$,\be\label{trilinear:thm}\langle\vv',\cP(\vv\cn\vv')\rangle_\scrH\lsm \|\vv\|_{\H^m}\|\vv'\|_{\H^{m-3}}^2\quad\text{ for }\vv\in\H^m\bigcap\sX,\;\;\vv'\in\H^{m-2}\bigcap\sX.\ee
Here, we used the $\H^{m-3}$ norm instead of $\scrH$ norm in the RHS, thanks to \eqref{H:H:equiv}. Indeed, by definition \eqref{def:scrH}, we estimate   the $\LL$ component and the higher derivative component  respectively in the above inequality.
The   $\LL$ component is simply zero as we can use the $\LL$-orthogonality of $\cP,\cQ$ and $\vv'\in\sX=\textnormal{image}\,\cP$ to get,
\be\label{v:cP:v:v}  \langle\vv', \cP(\vv\cn\vv')\rangle_\LL=\langle\vv',   \vv\cn\vv' \rangle_\LL=\intO{1\over2}\vv\cn|\vv'|^2 =0.\ee
Here, the last equality is due to the Divergence Theorem and $\vv\in \sX\bigcap \H^m$. Also note from \eqref{trilinear:thm} that $\vv'\in \H^{m-2}$, causing no regularity problem. 
 
 For the higher derivatives of \eqref{trilinear:thm}, it is needed only for $m\ge4$. Upon taking curl of \eqref{def:b:2:3} and using $\nt(\cP\vv)=\nt\vv$ given in \eqref{cKcq}, one has
\[\nt\cP(\vv\cn\vv')=\vv\cn(\nt\vv')+\rb_p(\nabla\vv,\nabla\vv')\]
for some bilinear \emph{polynomial} $\rb_p(\cdot,\cdot)$. By \eqref{b:H:3} ($k=m-4$),
\[\|\rb_p(\nabla\vv,\nabla\vv')\|_{\H^{m-4}}\lsm\| \vv\|_{\H^{m}}\| \vv'\|_{\H^{m-3}}.\]
 Also by \eqref{b:H:3} ($k=0$),
 \[\sum_{j=0}^{m-5}\left\|\nabla^{j}\big[\vv\cn(\nt\vv')\big] \right\|_\LL+\left\|\nabla^{m-4}\big[\vv\cn(\nt\vv')\big]-\vv\cn(\nabla^{m-4}\nt\vv')\right\|_\LL\lsm\| \vv\|_{\H^{m}}\| \vv'\|_{\H^{m-3}}.\]
Combining the above 3 lines of calculations, we obtain
\[\left\langle\nt\vv',\nt\cP(\vv\cn\vv')\right\rangle_{\H^{m-4}}\lsm \| \vv\|_{\H^{m}}\| \vv'\|^2_{\H^{m-3}}+\left\langle\nabla^{m-4}\nt\vv',\vv\cn(\nabla^{m-4}\nt\vv')\right\rangle_{\LL}.\]
But the last term is simply zero by the Divergence Theorem and $\vv\in \sX\bigcap \H^m$. Therefore, \eqref{trilinear:thm} hence \eqref{es:b:3} is confirmed.
 
Having validated all assumptions of Lemma  \ref{thm:linear:es}, we arrive at, with $M:=\eesup\bigl\{\| \vl\|_{\H^{m}}  \bigr\}$, \be\label{thm:lemma}\begin{aligned} \eesup \|\vl-\vp \|_{\scrH } &\le \eesup\left\|\ave{\rB{ \vp, \vq}}\right\|_{\scrH}\\&+{e^{CMT}-1\over CM}\eesup\Bigl\{ \left\|\rb(\vl,\ave{\rB{ \vp, \vq}})+\rb(\vl-\vp,\vp)\right\|_{\scrH} \Bigr\}.\end{aligned} \ee
To estimate the second last bilinear term, apply \eqref{b:H:3} ($k=m-3$)
\[\begin{aligned}&\left\|\rb(\vl,\ave{\rB{ \vp, \vq}})\right\|_{\scrH} \lsm  \left\| \vl\cn\ave{\rB{ \vp, \vq}} \right\|_{\H^{m-3}}\lsm \left\|\vl\right\|_{\H^{m }}\|\ave{ \rB{\vp,\vq }}\|_{  \H^{m-2} }.\end{aligned}\]
Comparing the last $\H^{m-2}$ norm with Lemma \ref{thm:ave:cb}, we see why we had to choose $\H^{m-3}$ norm in this current theorem.

To estimate the   last bilinear term of \eqref{thm:lemma}, apply \eqref{b:H:3} (with $k=m-3$ but $f=\nabla\vp$, $g=\vl-\vp$)
\[\left\|\rb(\vl-\vp,\vp)\right\|_{\scrH}\lsm \left\|(\vl-\vp)\cn\vp \right\|_{\H^{m-3}}\lsm \left\| \vp\right\|_{\H^{m }}\left\|\vl-\vp\right\|_{\H^{m-3}}.\]
Note it is crucial to only use $\H^{m-3}$ norm of $(\vl-\vp)$ because it is needed to bring \emph{closure} to the $\H^{m-3}$ estimate started from LHS of \eqref{thm:lemma}.

Substitute the above two estimates into \eqref{thm:lemma}, use the equivalence \eqref{H:H:equiv} and rearrange to obtain, with $M=\eesup\bigl\{\| \vl\|_{\H^{m}}  \bigr\}$,  $M^p:=\eesup\bigl\{\| \vp\|_{\H^{m}}  \bigr\}$,
\be\label{es:almost}\left(1-c_3{e^{CMT}-1\over CM}M^p\right)\eesup\|\vl-\vp\|_{\H^{m-3}}\le \left(c_1+c_2(e^{CMT}-1)\right)\eesup \|\ave{ \rB{\vp,\vq }}\|_{  \H^{m-2} } \ee
for some universal positive constants $c_1,c_2,c_3,C$. If   \be\label{choose:D2} MT, \,M^pT\le D_2,\ee for a constant $D_2$, then by the Mean Value Theorem,
\[{e^{CMT}-1\over CM}M^p={e^{CMT}-1\over CMT}TM^p\le e^{CD_2}D_2 .\]
Thus, we can choose $D_2$ so that the coefficient in the LHS of \eqref{es:almost} is no less than 1/2, yielding
\[\eesup\|\vl-\vp\|_{\H^{m-3}}\le 2\,\left(c_1+c_2(e^{CD_2}-1)\right)\eesup \|\ave{ \rB{\vp,\vq }}\|_{  \H^{m-2} }. \]

Combining it with    Lemma \ref{thm:ave:cb} and \eqref{choose:D2},  we conclude the proof.
  \end{proof}

 \section{\bf  Estimates  Independent of $\ep$}\label{sec:a:priori}
What does it mean to have estimates independent of $\ep$? Basically, we want calculations to be invariant under the hyperbolic scaling $\rho\to c\rho,\vv\to c\vv,t\to t/c,\ep\to \ep/c$ for any constant $c>0$. In Theorem \ref{thm:uni:es} for example, we will show $\H^m$ solutions exist at least for  time interval   of order $1/\|(\rho_0,\vv_0)\|_{\H^m}$. During this time interval,   the solution's $ {\H^m }$ norm is  at most inflated  by a   constant   and very importantly, the $\H^{m-1}$ norm of its first time derivative   is also only inflated by a constant. The latter estimate can be loosely stated as    ``what starts well-prepared, stays well-prepared''. Recall that having $\pt(\rho,\vv)\sim O(1)$ is equivalent to having $(\nc\vv,\nabla\rho)\sim O(\ep)$ as suggested  in \eqref{E0:E1}

The main difficulty is, even though  $\cL$ is skew-self-adjoint, namely $\displaystyle\intO \bpm\rho\\\vv\epm\cdot\cL\bpm\rho\\ \vv\epm\,dx=0$ for   $\vv\ccdot\nm\bc=0$, it is in general not the case for the spatial derivatives, namely $$\displaystyle\intO \pa_{x}^\beta \bpm\rho\\\vv\epm\cdot\pa^\beta_{x}\cL\bpm\rho\\ \vv\epm\,dx\ne0\quad\text{   if }\;\; \beta\ne0\text{ and }\pa\Omega\ne\emptyset.$$ This would've introduced $O(\ep^{-1})$ terms in the energy estimate. In addition, a very sophisticated mollification procedure would've been needed in estimating the highest spatial derivatives \cite{Rauch:char}.

On the other hand, we can recruit higher time derivatives since $\pa\Omega$ is static and,
\be\label{pt:cL}\intO \pa_{t}^k \bpm\rho\\\vv\epm\cdot\pa^k_{t }\cL\bpm\rho\\ \vv\epm\,dx=0\qquad\text{if \;\;}\pa^k_t \bpm\rho\\\vv\epm\in \H^1(\Omega)\text{ and } \vv\ccdot\nm\bc=0. \ee
 If one considers $[0,T]\times\pa\Omega$ as the lateral  boundary of the time-space domain, then the $\pt$ derivatives are precisely taken in the tangential directions, resonating with   the argument in  \cite{Rauch:char} that, near a characteristic boundary,      tangential and normal derivatives   are estimated differently.

 
 We first rescale the original system  into an equivalent one without explicit dependence on $\ep$. At the end of this section,   we   scale it back to the original formulation for which the essential results  will still be independent of $\ep$, so long as hyperbolic scaling is respected in the estimates.

 Recall the pressure law  $ \pres(\rho^\tot)=\pres(1+\ep\rho)$  rescaled to satisfy   $\pres(1)= \pres'(1)=1$.  Introduce new variable $r$ satisfying  \be\label{def:r} \pres(1+\ep\rho)= 1+\ep r  \iff  \pres^{inv}(1+\ep r)= 1+\ep \rho =\rho^\tot\ee
 where ${\pres}^{inv}$ denotes the functional inverse of $\pres$. Note by Taylor expansion $r=\rho+O(\ep\rho^2)$.  Then, we reformulate  the Euler equations \eqref{Euler:rhtot}   term-by-term,
\[\left\{\begin{split} {\ep\over \pres'({\pres}^{inv}(1+\ep r))}( \pt r+\vv\cn r )+ {\pres}^{inv}(1+\ep r) {\nc\vv }&=0,\\
 \pt\vv+\vv\cn\vv +{1\over {\pres}^{inv}(1+\ep r)}{\nabla r\over\ep}&=0.\end{split}\right.\]
Introduce\be\label{hyp:scale}V:=(  \brr,  \brv):=(\ep r,\ep \vv),\qquad \tau:=t/ \ep,\ee
 and rewrite the previous system in terms of $V$   as a symmetric hyperbolic PDE system, \be\label{sys:V} \ptau V+\brv\cn V=-  \DM({\brr}){\cL}(V),\qquad   \brv\ccdot\nm\bc=0\ee
where diagonal matrix\[ \DM(\brr):=\text{diag}\bigl\{{\pres}^{inv}(1+ \brr )\pres'({\pres}^{inv}(1+ \brr )), {1\over{\pres}^{inv}( 1+\brr )} ,\ldots,{1\over{\pres}^{inv}(1+ \brr )} \bigr\}.\] 
This rescaled, $\ep$-free system will be the main subject of this section.
 
Since   $\DM(0)=I$ and $\DM(\brr)$ is $\bC^\infty$,     there exists a constant $\brR$ only depending on $\pres(\ccdot)$ such that
\be\label{es:DM}\text{for }|\brr|\le \brR,\quad\text{ diagonal matrice  }\DM(\brr)\in[1/2,2] \ee
For example, if $\pres(\ccdot)$ satisfies the $\gamma$-power law, then by the scaling assumption \eqref{pres:scale}, it must be $\pres(\rho^\tot)= (\gamma-1+(\rho^\tot)^\gamma )/\gamma$ so that ${\pres}^{inv}(1+\brr)=(1+\gamma \brr )^{1/\gamma}$. Because physics suggests $\gamma\in[1,2]$, it is easy to calculate $\brR= (1-2^{-\gamma}) /\gamma$.

Introduce a mixed norm for $V(\tau,x)$ at any fixed time $\tau$
\[\nmm{V}_m(\tau):=\Bigl(\sum_{k=0}^m\|\ptau^k V\|^2_{\H^{m-k}(\Omega)}(\tau)\Bigr)^{1/2}.\] 
It  encapsulates all mixed space-time derivatives up to order $m$.

For brevity,  we  make a priori assumption throughout this section that 
\be\label{apriori}\text{integer } m \ge3, \;\; V\in\bC([0,\tau^*],\H^m(\Omega)),\;\;  \nmm{V}_m(\tau)\le1 \text{\;\; and \;\;}  |\brr(\tau,x)|\le\brR\ee
for all  $\tau\in[0,\tau^*]$,
 unless specified otherwise. The last two inequalities are not stringent at all because   by the scaling of $V$   in \eqref{hyp:scale},  \eqref{sys:V}, one expects  $\nmm{V}_m$ and $|\brr|$ to be $O(\ep)$ as long as the original unknown satisfies scaling $\|(r,\vv)\|_{\H^m}\sim O(1)$.   
 
 \subsection{Mixed norms}\label{subs:norms}

 First, we establish some basic facts of $\nmm{\cdot}$ as  direct consequences of Sobolev inequalities. Given  functions $f_1(\tau,x), f_2(\tau,x),\ldots,f_j(\tau,x)$ and a product of mixed mixed derivatives $ (\pa^{\beta_1}_{\tau,x}f_1)\,(\pa^{\beta_2}_{\tau,x}f_2)\ldots(\pa^{\beta_j}_{\tau,x}f_j)$, 
we have, \be\label{prod:ineq} \Bigl\|\prod_{i=1}^j\pa^{\beta_i}_{\tau,x}f_i \Bigr\|_{\LL}\lsm \prod_{i=1}^j\nmm{f_i}_k, \ee
if   one of the following conditions hold
 \[\left\{\begin{aligned}&|\beta_1+...+\beta_j|\le k,\;\;\text{ and }\;\; k\ge2,\;\;\text{ or }\\
 &|\beta_1+...+\beta_j|= k+1,\;\; \max\{|\beta_1|,\ldots,|\beta_j|\}\le k,\;\;\text{ and }\;\; k\ge3.\end{aligned}\right.\]
 The proof is exactly the same as the one given in the Appendix where only spatial derivatives are involved.

Estimates on $\DM(\brr)$ and its matrix inverse $\Din(\brr)$ will also be needed. By \eqref{es:DM},   $\brr$-derivatives of $\DM$ and $\Din$ can be bounded by constants only depending on  pressure law $\pres(\ccdot)$ and the order of derivatives, i.e. \be\label{es:DM:n}  \left|{d^k \over d\brr^k}\DM(\brr)\right|+\left|{d^k \over d\brr^k}\Din(\brr)\right|\lsm1\text{\quad for }\;\;|\brr|\le\brR.\ee
Also, by the Mean Value Theorem and  $\DM(0)=I$,
\be\label{es:DM:I}\left| \DM(\brr)-I\right|+\left|\Din(\brr)-I\right|\lsm |\brr|\text{\quad for }\;\;|\brr|\le\brR.\ee

Now,  inductively apply the chain rule and product rule to obtain, for multi-index $\beta$,
\[\begin{split}\pa^\beta_{\tau,x}\DM(\brr)&=\text{linear combination of }\,  (\pa^{\beta_1}_{\tau,x}\brr)\,(\pa^{\beta_2}_{t,x}\brr)\ldots(\pa^{\beta_j}_{\tau,x}\brr)\,{d^j \over d\brr^j}\DM(\brr)\\&\text{ over all   integers }j\in[1,  |\beta| ] \text{ and multi-indices satisfying  } \;\beta_1+...+\beta_j=\beta. \end{split}\]  
Therefore, by     \eqref{prod:ineq},  \eqref{es:DM:n} and the assumptions $\nmm{V}_m\le1$, $m\ge3$ set in \eqref{apriori},  \[\|\pa^\beta_{\tau,x}\DM(\brr)\|_{\LL}\lsm\nmm{\brr}_m, \text{ \qquad for  } 1\le|\beta|\le m.\]
Obviously, this estimate works for matrix inverse $\DM^{-1}$ as well. Sum  such estimates for all $|\beta|$ from 1 to $m$, and use \eqref{es:DM:I} for the $|\beta|=0$ case to arrive at 
\be \label{es:DM:mm}\nmm{\DM(\brr)-I}_m+ \nmm{\DM^{-1}(\brr)-I}_m \lsm\nmm{\brr}_m   .\ee
Similar estimate works for   $\ptau\DM$. In fact, applying \eqref{prod:ineq} to ${\ptau\DM =({d\over d\brr}\DM(\brr)-{d\over d\brr}\DM(0))\ptau\brr+{d\over d\brr}\DM(0)\ptau\brr}$, 
and  noting that  ${{d\over d\brr}\DM(\brr)-{d\over d\brr}\DM(0)}$ can be bounded in a way similar to  \eqref{es:DM:mm} , namely, \\
${\nmm{{d\over d\brr}\DM(\brr)-{d\over d\brr}\DM(0)}_{m-1}\lsm \nmm{\brr}_{m-1}\le1}$, we have\be \label{es:ptau:DM:mm}\nmm{\ptau\DM(\brr) }_{m-1}+ \nmm{\ptau\DM^{-1}(\brr) }_{m-1} \lsm\nmm{\ptau\brr}_{m-1}   .\ee

Now,   thanks to the   $\ep$-free formulation of \eqref{sys:V},   combine \eqref{es:DM:mm}, \eqref{es:ptau:DM:mm} with   Sobolev inequalities \eqref{prod:ineq} to inductively estimate $\ptau V,\ptau^2V,\ldots,\ptau^mV$ and obtain,  under the a priori assumption   \eqref{apriori}\be\label{nmm:equiv}   \nmm{\ptau V}_{m-1}\lsm  \mm{\ptau V}{m-1}\le \nmm{V}_m\lsm \|V\|_{\H^m}\ee
 
 \subsection{Vorticity estimates} 
Define $\brom:=\nabla\times\brv$. Take  $\nabla\times$ of the momentum equation in \eqref{sys:V},
\be\label{om:bB}\ptau\brom+\bB(\brv,\brom)=0\ee
 \[\text{where}\qquad
\bB(\vv,\om)=\begin{cases}\vv\cn\om+(\nc\vv)\om&\text{in   2D  }\\
\vv\cn\om+(\nc\vv)\om-\om\cn\vv&\text{in   3D }\end{cases}\]
 
 It is then an     exercise of energy estimates with Sobolev inequalities  to show (e.g. \cite[\S 17.3]{Taylor:III}),  with   any $\brv$ satisfying \eqref{apriori} and  $\vv\ccdot\nm\bc=0$,
\be\label{es:brom}\qquad \|\brom\|^2_{\H^{m-1}}\ttdiff   \lsm  \ttint\|\brom\|_{\H^{m-1}}^2 \|\nabla_x\brv\|_{\H^{m-1}}
\,\qquad \text{ for }\;\;0\le\tau_1<\tau_2\le\tau^*.\ee
Note that, due to $\brom\in\H^{m-1}(\Omega)$, before taking   the $(m-1)$th spatial derivatives of \eqref{om:bB}, one should mollify $\brom$ by first extending it outside $\pa\Omega$, then convolving it with a smooth kernel and restricting it back to $\Omega$   --- cf. \cite[pp. 489]{Taylor:III}.  Velocity $\brv$ needs not to be mollified because  it is assumed to be $\H^m(\Omega)$. It also means $\brv\big|_{\pa\Omega}$ is a defined function,   allowing us to utilize the $\brv\ccdot\nm\bc=0$ condition and Divergence Theorem to show such identities as $$\intO (\brv\cn f)f\,dx=-{1\over2}\intO(\nc\brv)|f|^2\,dx\;\quad\text{ for }\quad\; f\in \H^1(\Omega).$$

Next, recall \eqref{elliptic:bc} which  gives a bound for $V$ using $\cL(V)=( \nc\brv, \nabla\brr )$ and $ \nt\brv $,\be\label{elliptic:LK}
 \|V\|_{\H^k}\lsm  \|\cL(V)\|_{\H^{k-1} }+ \|\nt\brv \|_{\H^{k-1}  }+ \|V\|_{\LL },\quad\text{ if $\brv\ccdot\nm\bc=0$  and }  k\ge1.\ee
So it remains to estimate $\cL(V)$. In light of \eqref{pt:cL} and the fact that $\cL(V)$ and $\ptau V$ are connected via \eqref{sys:V}, we now move on to estimate $\|\pt^k V\|_{\LL}$ for $k=0,1,\ldots m$.

\subsection{Diagnostic estimates and recurrence}
We now connect mixed norms of $\cL(V)$ and $\ptau V$  via rescaled system \eqref{sys:V} and its time derivatives at any   fixed time. We call such   estimates ``diagnostic'', as opposed to ``prognostic'' estimates such as \eqref{es:brom} for $\brom$ in the form of time-dependent integral and differential  inequalities. Diagnostic estimates do NOT rely on evolutionary properties of \eqref{sys:V} and will not involve Gronwall type inequalities. They instead come from algebraic manipulation of \eqref{sys:V} at a fixed time $\tau$, typically using the product rule and Sobolev inequalities.

Elliptic estimate  \eqref{elliptic:LK} and the fact that $\ptau^k\brv\ccdot\nm\bc=0$  for $ k \in[0,m-1]$ imply,\[\|\Vk\|_{\H^{m-k }}\lsm\|\cL(\Vk)\|_{\H^{m-k -1}}+ \|\nt\ptau^{k }\brv \|_{\H^{m-k -1}}+\|\Vk\|_{\LL  }.\]
 On the other hand,
take $\ptau^k$ derivative on \eqref{sys:V} to get,   for $k\in[0,m-1]$,
\[  -\cL(\Vk)= \Vkl +\ptau^k(\brv\cn V)+ \ptau^k\big((\DM-I) \cL(V) \big). \]  
Combine these two    to obtain a recursive inequality, for every $ k \in[0,m-1]$,\be\label{recur:k:k1}\|\Vk\|_{\H^{m-k }}\lsm \|\Vkl\|_{\H^{m-k-1}}  +  \mm{\ptau^k\Bigl( \brv\cn V,\,(\DM-I) \cL(V),\,\brom\Bigr)}{m-k-1} +\|\Vk\|_{\LL  }.\ee
At one end of this recursive chain is $\mm{V}{m}$, the desirable norm, and at the other end is $\mm{ \Vm}{0}$, which will be estimated prognostically using  energy method in the next subsection.

Now,   exclude $k=0$ and connect from $k=1$ to $m-1$ in \eqref{recur:k:k1}, using the definition of $\nmm{ \cdot}$ and replace $\ptau\brom$ with the bilinear term from \eqref{om:bB} to get
$$\|\ptau V  \|_{\H^{m-1 }}\lsm     \nmm{\ptau\Bigl( \brv\cn V,\,(\DM-I) \cL(V) \Bigr) }_{m-2} +  \nmm{\bB(\brv,\brom)   }_{m-2}+\sum_{k=1}^m\|\Vk \|_{\LL  }$$
Combine it with   Sobolev inequalities \eqref{prod:ineq},  bounds on $(\DM-I)$ in \eqref{es:DM:mm} and the equivalence of $\nmm{\ccdot}_m$ and $\mm{\ccdot}m$ in \eqref{nmm:equiv} to reach the following lemma.

\begin{lemma}\label{thm:diagn:1} {\bf(Diagnostic estimates on first time derivative)} Consider $V(\tau,x)$, a solution of \eqref{sys:V} in the a priori setting of \eqref{apriori}. Then, 
\[\|\ptau V  \|_{\H^{m-1 }}\lsm     \mm{V }{m }^2+\sum_{k=1}^m\|\Vk  \|_{\LL  }\]
\end{lemma}
The extra quadratic term will cause no trouble, due to the scaling argument below \eqref{apriori}. 
 
For the $k=0$ case of  \eqref{recur:k:k1},  upon applying the same   \eqref{prod:ineq},   \eqref{es:DM:mm}, \eqref{nmm:equiv}, one obtains $\|V\|_{\H^m}\lsm  \|\ptau V  \|_{\H^{m-1 }}+\mm{\brom }{m-1}+  \mm{V }{m }^2+\|V\|_\LL$.   Combine it with Lemma \ref{thm:diagn:1} to get
\[2d_1\|  V  \|_{\H^{m  }}\le  \mm{\brom }{m-1}+  \mm{V }{m }^2+\sum_{k=0}^m\|\Vk  \|_{\LL  }\]
for some constant $d_1>0$. 
Then,  relax one of the  $\mm{V}{m}$ factors on the RHS to $ d_1$ and   absorb the associated quadratic term into the LHS, proving,  
 \begin{lemma}\label{thm:diagn} {\bf(Diagnostic estimates)}
Consider $V(\tau,x)$, a solution of \eqref{sys:V} in the a priori setting of \eqref{apriori}. If furthermore $\mm{V}m\le d_1$ for some constant $d_1$   solely depending on $m,\Omega$ and pressure law $\pres(\ccdot)$, then,  
\[ d_1\|  V  \|_{\H^{m  }}\le  \mm{\brom }{m-1}+      \sum_{k=0}^m\|\Vk  \|_{\LL  }.\]
  \end{lemma}
  
  \subsection{Prognostic estimates of $\|\ptau^k V\|_{\LL}$} Take  the     $\ptau^k$ derivative of \eqref{sys:V} and  single out the highest derivatives
\be\label{def:R:k}R^{(k)}:= -\sum_{j=1}^k{k\choose j}(\ptau^j\brv)\cn\ptau^{k-j}V+{k\choose j} (\ptau^j\DM)\cL(\ptau^{k-j}V)=\ptau(\Vk)+\brv\cn\Vk+ \DM\cL(\Vk). \ee
\begin{proposition}\label{thm:comm} (Diagnostic estimates of commutator)
Given a solution $V(\tau,x)$ to \eqref{sys:V} satisfying a priori assumption \eqref{apriori}. Then, for all  $k\in[0,m]$, 
\[ \begin{split}\|R^{(k)} \|_{\LL}&\lsm\mm{ \ptau V }{m-1} \mm{ V}{m}.\end{split}\] 
\end{proposition}\begin{proof}The case $k=0$ is trivial, so we consider $k\in[1,m]$.
Notice  that,  in the definition of $R^{(k)}$ in \eqref{def:R:k}, every product contains a factor as $\ptau\DM$, $\ptau\brv$ or their higher derivative, and another factor as $\nabla V$ or its higher derivative. Therefore, by  Sobolev inequalities \eqref{prod:ineq} with $k=m-1$,
\[\|R^{(k)} \|_{\LL}\lsm\nmm{(\ptau\DM,\,\ptau\brv)}_{m-1}\nmm{\pa_x V}_{m-1}. \]
Combine it with \eqref{es:ptau:DM:mm} and \eqref{nmm:equiv} to conclude the proof.
\end{proof}  
Before  obtaining     estimates of $\|\pt^k V\|_{\LL}$ in the next lemma, we carry out some calculation  for any $W$ with the same number of components as $V$ and with regularity $ W \in  \H^1(\Omega)$, $\ptau W \in  \LL(\Omega)$.
\[\begin{split} &2\intO(\Din  W )\cdot(\ptau W +\brv\cn W )  \\
\stackrel{(a)}{=}&    \intO\ptau (\Din W \cdot W )   +  \intO\brv\cn  (\Din W \cdot W )-\Bigl(\intO(\ptau\Din) W \cdot W+\intO(\brv\cn\Din) W \cdot W\Bigr) \\
\stackrel{(b)}{=}&\intO\ptau (\Din W \cdot W ) +  \intO\brv\cn  (\Din W \cdot W )-\intO(\ptau\brr+\brv\cn\brr)({d\over d\brr}\Din W \cdot W )\\
\stackrel{(c)}{=}&\intO\ptau (\Din W \cdot W )-\intO(\nc\brv)  (\Din W \cdot W )-\intO(\ptau\brr+\brv\cn\brr)({d\over d\brr}\Din W \cdot W )\\
\stackrel{(d)}{=}&{d\over d\tau}\intO\Din W \cdot W -\intO\nc\brv ( \Din-\DM_{1,1}{d\over d\brr}\Din ) W \ccdot W. \end{split}\]
Here, $(a)$ is by the product rule and the fact that $\DM$ is diagonal,   $(b)$ is by the chain rule,  $(c)$ is by   the Divergence Theorem and     $\brv\cdot\nm\bc=0$, and $(d)$ is a simple substitution via the mass equation of \eqref{sys:V} with $\DM_{1,1}$ denoting the first entry of matrix $\DM$. 

Next, add and subtract a $2\intO (\Din W)\ccdot(\DM\cL(W))=2\intO W\ccdot\cL(W)$ term and rearrange it,
\[\begin{split}{d\over d\tau}\intO\Din W \ccdot W =&\,  -2\intO W\ccdot\cL(W) +\intO\nc\brv ( \Din-\DM_{1,1}{d\over d\brr}\Din ) W \ccdot W\\&+2\intO(\Din  W )\ccdot\Big(\ptau W +\brv\cn W+\DM\cL(W)\Big)\end{split}  \]
Use \eqref{es:DM:n} to bound the maxima of $ \DM ,\Din , \pa_{\brr} \Din$, and  also use    $|\nc\brv|_{\Linf}\lsm \mm{V}{m }$ to arrive at, for  $W \in  \H^1(\Omega)$ and $\ptau W \in  \LL(\Omega)$,
\be\label{Wn}\begin{split}\left|{d\over d\tau}\wL{W}^2\right|\lsm&\,  \left|\intO W\ccdot\cL(W) \right|+\mm{V}{m } \wL{W}^2 \\&+ \|\ptau W  +\brv\cn W +\DM\cL(W)\|_{\LL}\wL{W}.\end{split}  \ee
Here, \[\wL{W}:=\Bigl(\intO\Din  W  \ccdot W  dx\Bigr)  .\]
\begin{lemma}\label{thm:progn}{\bf(Prognostic estimates)}
Consider $V(\tau,x)$ a solution to \eqref{sys:V} satisfying a priori assumption \eqref{apriori}. Then, for all $k\in[0,m]$ and $0\le\tau_1<\tau_2\le\tau^*$,\[  \wL{ \Vk }^2 \ttdiff \lsm \ttint \Bigl(   \wL{\Vk} +\mm{\ptau V}{m-1}\Bigr) \wL{\Vk}\mm{V}{m } .\] 
\end{lemma}
Note  $k=m$ is included here. Also, for $k\ge1$, both sides are in some sense quadratic in $\ptau V$, which will eventually yield desirable bound on  the inflation of $\mm{\ptau V}{m-1}$.
\begin{proof}  
First, restrict the value of $k\in[0,m-1]$ so that $\Vk\in \H^1(\Omega)$ and   $\ptau^k\brv\ccdot\nm\bc=0$ is well-defined,    allowing us to apply the Divergence Theorem to have
\be\label{cL:Vn}\intO\Vk\ccdot\cL(\Vk)=0,\qquad\text{ if }k\in[0,m-1].\ee
Thus,  we set $W=\Vk$ in  \eqref{Wn} where we also     apply \eqref{cL:Vn}  to cancel out the first  term in the RHS   and     use     Proposition \ref{thm:comm}    to estimate the   last term to prove the lemma for $k\in[0,m-1]$.

Estimating the highest   derivative $\Vm$ requires more care because it is merely in $\LL(\Omega)$.   Thus,   \eqref{Wn} and   \eqref{cL:Vn} are not directly applicable here. One remedy is to apply mollification in time, namely a time filter, to increase time regularity. Here instead, we demonstrate the closely related time-averaging technique. 
Let $W$  in   \eqref{Wn} be the time average of $\Vm$, i.e. with small $\delta>0$,  \[W=W^\delta:={1\over\delta }\int_{\tau}^{\tau+\delta}\Vm={1\over\delta }\Vml\Big|_{\tau}^{\tau+\delta}\]Then,   $W^\delta \in\H^1(\Omega)$,  $\ptau W^\delta\in\LL(\Omega)$, so that \eqref{Wn} holds and also the Divergence Theorem  applies,\be\label{W:delta:0}\intO W^\delta\ccdot\cL(W^\delta)=0.\ee 
It remains to estimate the first factor of  the     last term in   \eqref{Wn}. At fixed time $\tau$, we have
\be\label{es:Vm}\begin{split}&\,\delta\ccdot\Bigl(\ptau W^\delta+\brv\cn W^\delta+\DM\cL(W^\delta)\Bigr)\\=&\,(\ptau \Vml+\brv\cn \Vml+\DM\cL( \Vml))\Big|_{\tau}^{\tau+\delta}\\&\,-\left(\brv\Big|_{\tau}^{\tau+\delta}\right)\cn \Vml(\tau+\delta,\cdot)-\left(\DM\Big|_{\tau}^{\tau+\delta}\right)\cL\big( \Vml(\tau+\delta,\cdot)\big) \end{split}\ee
In the RHS, the first term is indeed $\int_{\tau}^{\tau+\delta}\ptau R^{(m-1)}$. By  the same reasoning as used in   Proposition \ref{thm:comm}, we find   $\|\ptau R^{(m-1)}\|_\LL\lsm\mm{ \ptau V}{m-1} \mm{V}{m}$. The second last term has a factor $\brv\big|_{\tau}^{\tau+\delta}= \int_{\tau}^{\tau+\delta}\ptau\brv$ and we know    $|\ptau\brv|_{\Linf}\lsm\mm{\ptau \brv}{m-1}$. Similarly, the   last term has a factor $\DM\big|_{\tau}^{\tau+\delta}= \int_{\tau}^{\tau+\delta}\ptau\DM$ and we know  $|\ptau\DM|_{\Linf}\le |\DM'|_{\Linf}|\ptau\brr|_{\Linf}\lsm \mm{\ptau\brr}{m-1}$. Finally, notice that $\|\nabla\Vml\|_\LL+\|\cL(\Vml)\|_\LL\lsm \mm{V}{m}$ by \eqref{nmm:equiv}. Apply all these estimates  in the $\LL$ norm of \eqref{es:Vm} to obtain,
\[\begin{split}\delta\|\ptau W^\delta+\brv\cn W^\delta+\cL(W^\delta)\|_{\LL }\lsm&\int_{\tau}^{\tau+\delta}\mm{ \ptau V}{m-1}\max_{[\tau,\tau+\delta]} \mm{V}{m}   \end{split}\]
Now, apply this estimate to $\int_{\tau_1}^{\tau_2-\delta}$ of  \eqref{Wn} with $W=W^\delta$, apply \eqref{W:delta:0}  and   pass the limit as $\delta\to 0+$  to prove the lemma for $k=m$. Note by $V\in\bC([0,\tau^*],\H^m)$,   we have $\ptau^k V\in\bC([0,\tau^*],\H^{m-k})$, and thus $\displaystyle\lim_{\delta\to0}\bigl\|W^\delta-\Vm\bigr\|_\LL(\tau)=0$.\end{proof}

\subsection{Estimates of $\|(\rho,\vv)\|_{\H^m}$   and   $\|\pt(\rho,\vv)\|_{\H^{m-1}}$ } \label{subset:uniform}We will still mostly work with $V=\ep( r, \vv)$, only reconnecting with $\vv$ and $r\approx\rho$ in Theorem \ref{thm:uni:es} near the end of this section.
The goal to keep in mind is the existence time in terms of $\tau$ at the same order of $1/\|V_0\|_{\H^m}$ with the $\|V\|_{\H^m}$ norm only inflated by a constant. We also need to see similar inflation of $\mm{\ptau V }{m-1}$ but will tolerate some additional term that is {\it quadratic} in $\|V\|_{\H^m}$.


\begin{lemma}\label{thm:es:V}{\bf (Estimates on $V$ and $\ptau V$)} Consider a solution of \eqref{sys:V},    $V\in\bC([0,\tau^*],\H^m(\Omega))$ with $m \ge3$. Then, there exist positive constants  $\tau^\sharp,C_v,C_1,C_2,C_3 $ that solely depend on $m,\Omega$ and pressure law $\pres(\ccdot)$ so that, if  $\|V\|_{\bC([0,\tau^*],\H^{m} )}\le C_v$, then for  times $\tau\in[0,\tau^*]\bigcap[0, \tau^\sharp  /\mm{V_0}m]$,
\begin{align} 
\label{es:V:Hm}\mm{V }m&\le C_1  \mm{V_0}m\\
\label{es:V:Hm1}\mm{\ptau V }{m-1}&\le C_2 \mm{\ptau V_0}{m-1} +C_3  \mm{V_0}{m}^2\end{align}\end{lemma}
Notice, under the hyperbolic rescaling \eqref{hyp:scale}, one has $\mm{V}{m}=\ep\mm{(r,\vv)}{m}$ and $\mm{\ptau V}{m}=\ep^2\mm{\pt(r,\vv)}{m}$ and such substitution would not change the structure of \eqref{es:V:Hm}, \eqref{es:V:Hm1}.

\begin{proof} We pick constant $C_v$ so that $\|V\|_{\bC([0,\tau^*],\H^{m} )}\le C_v$ implies a priori assumption \eqref{apriori} and   $\mm{V}m\le d_1$ as required by Lemma \ref{thm:diagn}. 

Introduce the shorthand notations
\[\begin{aligned}&\qquad F(\tau):=\mm{V}{m}(\tau),&&f (\tau):=\mm{\ptau V}{m-1}(\tau),\\&\qquad \Phi(\tau) :=  \Bigl(\| \brom\|_{\H^{m-1}}^2+\wL{V}^2+\phi^2(\tau)\Bigr)^{1/2},&&\phi(\tau):=  \Bigl(\sum_{k=1}^m\wL{\Vk}^2\Bigr)^{1/2},\end{aligned}\]
where lowercase $f, \phi$   involve at least one  $\ptau$ derivatives and our eventual goal is to estimate $F,f$.

By definition and $|\DM|\in[1/2,2]$, we have  $\phi\le\sqrt2\nmm{\ptau V}_{m-1}$ and $\Phi\le\sqrt2\nmm{V}_m$. Combine it with \eqref{nmm:equiv} to get
 \begin{align}\label{G:F}   \Phi\lsm F,&\qquad  \phi\lsm f. \\
\intertext{Meanwhile, we have been gathering diagnostic estimates in Lemmas \ref{thm:diagn:1}, \ref{thm:diagn}, i.e.}\label{ddd}    F   \lsm\Phi,&\qquad  f \lsm (F^2+\phi)  ,  \\
 \intertext{and prognostic estimates in \eqref{es:brom} and Lemma \ref{thm:progn} which   sum up to }
\label{ppp}   \ttdiff \Phi^2 \lsm \ttint ( \Phi +f)\Phi F,& \qquad  \phi^2\ttdiff \lsm\ttint (\phi+f)\phi F.\end{align}

(i). Estimate of $F(\tau)$. In the first part of   \eqref{ppp},   relax $f$ to $F$ a la \eqref{nmm:equiv} and relax $F$ to $\Phi$  a la  \eqref{ddd} to obtain 
$  \Phi^2 \ttdiff   \le 2c_1  \displaystyle \ttint \Phi^3$ for some constant $c_1$. Thus,  $\Phi^2(\tau)\le \Phi^2(0)+2c_1\displaystyle\int_0^\tau \Phi^3$ and by the continuity of $\Phi(\tau)$ and the comparison principle,    $$\Phi\le \tilde{\Phi}\;\;\text{ solving }\;\; {d\over d\tau}(\tilde{\Phi} )^2=2c_1(\tilde{\Phi})^3,\quad \tilde{\Phi}(0)=\Phi(0) $$  
\[\implies \Phi(\tau)\le \tilde{\Phi}(\tau)={\Phi(0)\over 1-c_1  \Phi(0)\tau} \]
as long as the RHS is bounded.  Thus \[1-c_1  \Phi(0)\tau\ge1/2\implies \Phi(\tau) \le 2\Phi(0).\]
Then, by the equivalence of $F,\Phi$ as in the first parts of \eqref{G:F}, \eqref{ddd}, we proved \eqref{es:V:Hm} as well as the $\tau$ interval   prescribed above it.

(ii). Estimate of $ f(\tau)$.  Combine the second parts of  \eqref{ddd}, \eqref{ppp}, and relax $F(\tau)$ to $F(0)$ a la \eqref{es:V:Hm},\[\phi^2\ttdiff \le2c_2\ttint (\phi +F^2(0))\ccdot\phi\ccdot F(0).\]  
 By the continuity of $\phi(\tau)$ and the comparison principle, $$\phi\le  \tilde\phi\;\;\text{ solving }\;\;{d\over d\tau}( \tilde\phi)^2=2c_2( \tilde\phi+F^2(0))\ccdot \tilde\phi \ccdot F(0)  ,\quad \tilde\phi(0)=\phi(0)$$
 \[\implies\phi(\tau)\le \tilde\phi (\tau)= -F^2(0)+e^{c_2 F(0) \tau}(\phi(0)+F^2(0)). \]
Combine it with the $\tau$ interval above \eqref{es:V:Hm} and the second parts of \eqref{G:F}, \eqref{ddd} to prove \eqref{es:V:Hm1}.
 \end{proof}

 This lemma leads to the final theorem of this section.
\begin{theorem}\label{thm:uni:es}{\bf(Uniform estimates)}
 Under the same hypotheses as Main Theorem \ref{thm:main} with $\Eo,\Eto$ equivalently given in \eqref{E:rho} namely, $\Eo=\|(\rho_0,\vv_0)\|_{\H^{m}}$, $\Eto=\|\pt(\rho_0,\vv_0)\|_{\H^{m-1}}$, there exist constants $E^*,T^\sharp,C^\sharp,C_t^\sharp,C_\cL^\sharp $ that solely depend on $m,\Omega$ and pressure law $\pres(\ccdot)$ so that,  \begin{subequations}\be\label{uni:es:00}\Eo\le E^*/\ep\text{\;\; implies there exists a unique $\bC^1$ solution for\;\; }t\in [0,{T^\sharp/ \Eo} ].\ee
 More precisely,     \begin{align} \label{uni:es:0}
 \|(\rho,\vv)\|_{\bC([0,T^\sharp/\Eo],\H^{m} )}&\le C^\sharp \Eo,\\  \label{uni:es:1}
 \|\pt(\rho,\vv)\|_{{\bC([0,T^\sharp/\Eo],\H^{m-1} )}}&\le C^\sharp_t( \Eto+\Eo^2),\\  \label{uni:es:2}
 \|\cL(\rho,\vv)\|_{{\bC([0,T^\sharp/\Eo],\H^{m-1} )}}&\le C^\sharp_\cL \,\ep\,( \Eto+\Eo^2). \end{align}\end{subequations}
 \end{theorem}\begin{proof}The short time existence of classical solutions is established in \cite{Sch:Euler}, so we only prove the estimates here. The continuation method is always at our disposal, since the compatibility condition $(\pt^k\vv_0)\ccdot\nm\bc=(\ptau^k\brv_0)\ccdot\nm\bc=0$ is invariant under hyperbolic rescaling \eqref{hyp:scale}.
 
First of all, by   the close relation of $r$ and $\rho$ in \eqref{def:r} namely $\brr=\ep r=\pres(1+\ep\rho)-\pres(1)$ and $\ep\rho=\pres^{inv}(1+\ep r)-\pres^{inv}(1)$, we can use similar technique for proving \eqref{es:DM:mm}, \eqref{es:ptau:DM:mm} to show that
\begin{align*} \mm{\ep \rho }m\le1,\;\;|\ep\rho|\le1/2 &\implies\mm{ r}m\lsm\mm{ \rho }m,\;\;\mm{ \pt r }{m-1}\lsm\mm{ \pt \rho  }{m-1},\\
 \mm{\ep r }m\le1,\;\;|\ep r|\le\brR &\implies\mm{\rho}m\lsm\mm{r}m,\;\;\mm{ \pt\rho }{m-1}\lsm\mm{ \pt r  }{m-1}.\end{align*}
Therefore, by applying hyperbolic rescaling \eqref{hyp:scale} to the target conclusions, it suffices to show there exist    universal constants $e^*,\tau^\sharp ,c_0,c_1, c_2$, suitably chosen so that,\begin{subequations}
\be\label{uni:es:00V} \mm{V_0}{m}\le e^*\text{\;\; implies there exists a unique $\bC^1$ solution for\;\;  }\tau\in [0,\tau^\sharp /\mm{V_0}{m} ], \ee
and  that  \begin{align}\label{uni:es:0V}\|V\|_{\bC([0,\tau^\sharp /\mm{V_0}{m}],\H^{m} )}&\le c_0\mm{V_0}{m},\\ 
 \label{uni:es:1V}\|\ptau V\|_{ {\bC([0,\tau^\sharp /\mm{V_0}{m}],\H^{m-1} )}}&\le c_1( \mm{\ptau V_0}{m-1}+\mm{V_0}{m}^2),\\
 \label{uni:es:2V}\|\cL(V)\|_{ {\bC([0,\tau^\sharp /\mm{V_0}{m}],\H^{m-1} )}}&\le  c_2  ( \mm{\ptau V_0}{m-1}+\mm{V_0}{m}^2). \end{align}\end{subequations}
 
Indeed, choose $e^*=C_v/C_1$ with $C_1>1$ and $C_v$  used in Lemma \ref{thm:es:V}.    Then, by continuity argument and Lemma \ref{thm:es:V}, the a priori assumption $\mm{V}m\le C_v$ as well as \eqref{uni:es:0V}, \eqref{uni:es:1V} remain true in the time interval given in \eqref{uni:es:00V}. Finally, \eqref{uni:es:2V} is by a simple deduction from  \eqref{uni:es:0V}, \eqref{uni:es:1V}, the $\ep$-free formulation \eqref{sys:V}, Sobolev inequalities and bounds of $\DM$ in \eqref{es:DM:mm}.
 
  \end{proof}

  \section{\bf Proof of the Main Theorem and Concluding Remarks}\label{sec:conclude}
  Now we prove the Main Theorem \ref{thm:main} using the time-averaging estimates in Theorem \ref{thm:ave:es} and the $\ep$-independent estimates in Theorem \ref{thm:uni:es}.
  
  \begin{proof}(Main Theorem \ref{thm:main}). First, regarding the time interval of validity, by  \eqref{uni:es:0} of Theorem \ref{thm:uni:es},  replace $T^\flat$ with $T^\sharp/\Eo$ in the  last line of Theorem \ref{thm:ave:es}  to get \be\label{T:bounds}T\in\bigl[0,T^\sharp/\Eo\bigr]\bigcap \bigl[0,\,  \displaystyle{D_2/ \esssup_{[0,T^\sharp/\Eo]}\|( \rho, \vv ,\vl)   \|_{\H^{m } }}\bigr].\ee
  By  estimate \eqref{uni:es:0} again, and by a similar estimate well known to be true for $\vl$ (e.g. \cite[Ch.  17, Thm. 3.2]{Taylor:III}), we further shorten the second time interval to $[0, D_2/(C^\sharp \Eo)]$. Therefore, take   $T^*:=\min\{T^\sharp,D_2/C^\sharp\}$ to make both Theorem \ref{thm:ave:es} and \ref{thm:uni:es}   valid for $T\in[0,T^*/\Eo]$.
  
  On this time interval, Theorem \ref{thm:ave:es} guarantees
  \[\eesup\|\vl-\vp\|_{ \H^{m-3} } \lsm\,\ep\,\eesup\bigl\|\cL(\rho, \vv)\bigr\|_{\H^{m-1}} \esup\|\vp\|_{\H^m}\]
One $\ep$ factor  is   in place,  and \eqref{uni:es:2} of  Theorem \ref{thm:uni:es} guarantees another $\ep$ factor   from $\cL(\rho, \vv)$,
\be\label{main:almost}\sup_{t\in[0,{T }]}\|\vp-\vl\|_{\H^{m-3}}\lsm\,\ep^2\, (\Eto+\Eo^2)\esup\|\vp\|_{\H^m}.\ee

So, our last job is to bound $\|\vp\|_{\H^m}$ in terms of $\|\vp_0\|_{\H^m}$. Scale vorticity estimate \eqref{es:brom} back to variables $\om, t$ so that
$\|\omega\|^2_{\H^{m-1}} \Big|_{t_1}^{t_2}  \lsm\displaystyle\int_{t_1}^{t_2}  \|\vv\|_{\H^{m}}\|\omega\|^2_{\H^{m-1}}$. 
Apply energy method to  \eqref{UP:B}, noting     $\vp\ccdot\nm\bc=\vq\ccdot\nm\bc=0$, to get
\(\|\vp\|^2_{\LL }\Big|_{t_1}^{t_2}\lsm \displaystyle\int_{t_1}^{t_2}\|\vv\|_{\H^{m}} \|\vp\|_{\LL }^2 \).
Combine these two Gronwall inequalities with elliptic estimate \eqref{es:cp:L2} to obtain 
$\|\vp\|_{\H^m}(T)\le \|\vp_0\|_{\H^m}e^{C T\esup\|\vv\|_{\H^m}}.$
Then, the exponent can be relaxed  to a constant due to    \eqref{T:bounds}. 
The proof   is complete!
  \end{proof}
  
  The proof of   Corollary \ref{thm:C} is as follows. First, regarding  $\|\vv-\vl\|_{ \H^{m-3}}$,     combine  \eqref{uni:es:2} of  Theorem \ref{thm:uni:es} with elliptic estimates from Proposition \ref{prop:cQ} to obtain $\|\vq\|_{ \H^{ m}}\lsm  \ep(\Eto+\Eo^2).$  By the  $O(\ep^2)$ estimate of $\|\vp-\vl\|_{ \H^{m-3}}$ from the Main Theorem,  this is more than enough to prove the $O(\ep)$ estimate of $\|\vv-\vl\|_{ \H^{m-3}}$. 
  
  Secondly, regarding  $\intT \vv-\vl =\int_0^T{\vq}+ \int_0^T{\vp-\vl}$,   one starts with estimate \eqref{ave:vq:cL} of $\ave{\vq}$ and estimate   \eqref{uni:es:2} of $\cL(\rho,\vv)$ to obtain
  \be\label{es:vq:9}\|\int_0^T{\vq}\|_{\H^{m-1}}\lsm \ep^2(\Eto+\Eo^2)(1+T\eesup \|(\rho,\vv)\|_{\H^{m }}) . \ee
  Combine it with \eqref{main:almost}, \eqref{T:bounds} to complete the proof of \eqref{es:C}.
  
Lastly, regarding the transport equations in Corollary \ref{thm:C}, we rewrite them as
\[\begin{aligned}\pt\thl+\vl\cn\thl&=0,\\\pt\th +\vp\cn\th &=-\vq\cn\th,\end{aligned}\]
and fit it into the notations of Lemma \ref{thm:linear:es},    $$\begin{array}{lll} v_1=\vl,\;\;&u_1=\thl,\;\;&\si_1=0,\\ v_2=\vp,\;\;&u_2=\th,\;\;&\si_2=-\vq\cn\th,\\
 \scrB= \H^m\bigcap\sX,\;\;&\scrB'=\H^m,\;\;&\scrB''=\H^{m-1} \end{array}$$ with $$\rb(v,u)=v\cn u,\qquad\scrH=\H^{m-3}.$$
 Then,   simply application of Sobolev inequalities and the fact that $\vl\in\sX$ can validate all assumptions of Lemma \ref{thm:linear:es}, including \eqref{es:b:3}. Therefore, we have, with $M:=\eesup\bigl\{\| \vl\|_{\H^m}  \bigr\}$, \be\label{theta:almost} \begin{aligned}&\eesup \|\thl-\th \|_{ \H^{m-3} }- \eesup\|\ave{\vq\cn\th}\|_{\H^{m-3}}\\\le\,& {e^{CMT}-1\over CM}\eesup\Bigl\{ \big\|\rb(\vl,\ave{\vq\cn\th})+\rb(\vl-\vp,\,\th)\big\|_{\H^{m-3}} \Bigr\}\\\lsm\,&
 T  \|\vl\|_{\H^m}\|\ave{\vq\cn\th}\|_{\H^{m-2}}+T\|\vl-\vp\|_{\H^{m-3}} \|\th\|_{\H^{m}},\end{aligned} \ee
 where the last estimate is due to \eqref{T:bounds} and Sobolev inequalities.
 
To estimate the time-average $\ave{\vq\cn\th}$, we
perform integrating by parts  and use the transport equation of $\th$ itself,
\[\ave{\vq\cn\th}=\ave{\vq}\cn\th\Big|_0^T-\intT\ave{\vq}\cn\pt\th=\ave{\vq}\cn\th\Big|_0^T+\intT\ave{\vq}\cn(\vv\cn\th)\]
   Take its $\H^{m-2}$ norm, apply   \eqref{es:vq:9}  and Sobolev inequalities   to get
\be\label{es:xi:0} \|\ave{\vq\cn\th}\|_{\H^{m-2}}\lsm \ep^2(\Eto+\Eo^2)\esup\|\th\|_{\H^{m}}(1+T\esup \| \vv \|_{\H^{m}}).\ee
 Finally, the estimation of $\|\th\|_{\H^{m}}$ follows   the standard energy method  together with the same mollification for proving vorticity estimate \eqref{es:brom},  $$ \|\th\|_{\H^{m}}^2\Big|_{t_1}^{t_2}\lsm \int_{t_1}^{t_2}\|\vv\|_{\H^{m}} \|\th\|_{\H^{m}}^2 \implies \|\th\|_{\H^m}(T) \le \|\th_0\|_{\H^m}e^{CT\esup \| \vv \|_{\H^{m}}}.$$  
 Therefore, in \eqref{theta:almost}, apply this estimate together   with \eqref{es:xi:0}, \eqref{main:almost} and \eqref{T:bounds}   to finish the proof of the very last inequality of Corollary \ref{thm:C}.

 For future studies, we   like to comment on the possibilities of sharpening the error estimates for practical use such as numerical analysis, because the incompressible approximation is   ubiquitously important. One aspect is to get some good bounds on the inequality constants $C^*,T^*$ etc. This can benefit from using optimal constants in the Sobolev inequalities, making all $\lsm$ relations explicitly $\le$ relations. In addition, for the easier case     $\pa\Omega=\emptyset$, one can drastically reduce the steps of the energy method  in Section \ref{sec:a:priori}, potentially reducing constants as well. Another aspect is to utilize dispersive and/or dissipative mechanisms  which   the current article does not reply on. It will be very interesting to see what role they can play when combined with time-averaging. 
  
 Furthermore, we note that it is   easy to extend our techniques to domains living in two and three dimensional Riemannian manifolds, in which case two major      analytical tools   remain valid: Stokes' theorem as generalization of Divergence Theorem    and Sobolev inequalities. Also, calculations carried out in this article  mostly rely on a handful of coordinate-independent  operators, i.e. $\nabla, \nabla\ccdot,\nt,\vv\cdot\!\nabla,\Delta$. Then, our results   and techniques can be applied to  interesting areas  such as geophysical fluid dynamics on a sphere and relativistic fluid dynamics.

\section*{\bf Appendix}
We prove the Sobolev (type) inequality  as described in Remark \ref{re:Sobolev}, that is, on a smooth, compact domain in two or three dimensions,
\be\label{App:Sob}\Bigl\|\prod_{i=1}^j\px^{\beta_i}f_i \Bigr\|_{\LL}\lsm \prod_{i=1}^j\Bigl\|f_i\Bigr\|_{\H^k} \ee 
if   one of the following conditions hold
 \[\left\{\begin{aligned}&|\beta_1+...+\beta_j|\le k,\;\;\text{ and }\;\; k\ge2,\;\;\text{ or }\\
 &|\beta_1+...+\beta_j|= k+1,\;\; \max\{|\beta_1|,\ldots,|\beta_j|\}\le k,\;\;\text{ and }\;\; k\ge3.\end{aligned}\right.\]

For convenience, assume $k\ge|\beta_1|\ge|\beta_2|\ge\cdots|\beta_j|$.  First, it is an easy exercise of proof by contradiction (for example) to show that, under the above hypotheses,  we must have
\be\label{beta:3}   |\beta_3|\le k-2,\ee
and one of the following scenarios  must occur
 \be\label{beta:1:2}\left\{\begin{aligned}&|\beta_2|\le k-2,\;\;\text{ or }\\&|\beta_1|= |\beta_2|= k-1.\end{aligned}\right.\ee

Next, by H\"older's inequality,
\[\Bigl\|\prod_{i=1}^j\px^{\beta_i}f_i \Bigr\|_{\LL}\le\Bigl\| \px^{\beta_1}f_1\px^{\beta_2}f_2 \Bigr\|_{\LL} \prod_{i=3}^j\Bigl|\px^{\beta_i}f_i\Bigr|_{\sL^\infty}.\]
Then, by \eqref{beta:3} and Sobolev embedding $\H^2\subset \sL^\infty$, we can relax all the above $\sL^\infty$ norms to $\H^k$ norms and therefore, the proof of \eqref{App:Sob} is reduced to proving
\be\label{App:1:2}\Bigl\| \px^{\beta_1}f_1\px^{\beta_2}f_2 \Bigr\|_{\LL}\lsm \Bigl\|f_1\Bigr\|_{\H^k} \Bigl\|f_2\Bigr\|_{\H^k}\quad\text{ if one of \eqref{beta:1:2} holds}.  \ee
 
Under the first scenario of \eqref{beta:1:2}, we simply apply as before the same   combination of H\"older's inequality and Sobolev embedding $\H^2\subset \sL^\infty$ to finish the proof of \eqref{App:1:2}. Under the second scenario of \eqref{beta:1:2}, apply a different version of H\"older's inequality
\[\Bigl\| \px^{\beta_1}f_1\px^{\beta_2}f_2 \Bigr\|_{\LL}\le \Bigl\| \px^{\beta_1}f_1\Bigr\|_{\sL^4}\Bigl\|\px^{\beta_2}f_2 \Bigr\|_{\sL^4}\]
and a different version of Sobolev embedding $\H^1\subset \sL^4$ to finish the proof of \eqref{App:1:2}. 

\section*{\bf Acknowledgement}

B. Cheng would like to thank I. Gallagher for very insightful discussion while both were participating the ``Mathematics for the Fluid Earth'' programme at the Isaac Newton Institute, Autumn 2013.  BC   thanks INI for its hospitality  and   is very grateful to the organizers of this programme for offering a Visiting Fellowship.

BC is also thankful for  stimulating discussions with and expertise from  H. Beir\~{a}o da Veiga,  M. Cullen, S. Jin,  M. Oliver, I. Roulstone,  P. Secchi, S. Schochet, E. Tadmor, M. Tang, C. Xie,  B. Wingate,  S. Zelik, C. Zeng and the anonymous referee(s).

BC gratefully acknowledges administrative support by the Department of Mathematics and the University of Surrey, which   particularly expedited  the accomplishment of this work. 
 
  \end{document}